\documentclass{amsart}
\usepackage{lineno}
\modulolinenumbers[5]
\usepackage{setspace}
\usepackage{hyperref}
\usepackage{amsmath}
\usepackage{amssymb}
\usepackage{amsfonts}
\usepackage{bm}
\usepackage{amsthm}
\usepackage{array,epsfig,fancyhdr}
\usepackage[sectionbib,numbers]{natbib}
\usepackage[utf8]{inputenc}
\usepackage{CJKutf8}
\usepackage{tikz}
\usepackage{mathtools}
\usepackage{tikz}

\usepackage{xr}
\newtheorem{theorem}{Theorem}[section]
\newtheorem{lemma}[theorem]{Lemma}
\newtheorem{corollary}[theorem]{Corollary}

\newtheorem{fconjecture}[theorem]{False Conjecture}

\theoremstyle{definition}
\newtheorem{definition}[theorem]{Definition}
\newtheorem{example}[theorem]{Example}

\theoremstyle{remark}
\newtheorem{remark}[theorem]{Remark}

\numberwithin{equation}{section}

\DeclareGraphicsExtensions{.pdf,.png}
\usepackage{wrapfig}
\usepackage{lipsum}

\begin{document}

\title[To prove the Four Color Theorem unplugged]{A renewal approach to prove\\ the Four Color Theorem unplugged\\[1.5ex]{\footnotesize Part II:}\\[0.5ex] R/G/B Kempe chains in an extremum non-4-colorable MPG}

\author{Shu-Chung Liu}
\address{Institute of Learning Sciences and Technologies, National Tsing Hua University, Hsinchu, Taiwan}
\email{sc.liu@mx.nthu.edu.tw}



\date{\today}



\begin{abstract}
This is the second part of three episodes to demonstrate a renewal approach for proving the Four Color Theorem without checking by a computer. The first and the third episodes have subtitles: ``RGB-tilings on maximal planar graphs'' and ``Diamond routes, canal lines and $\Sigma$-adjustments,'' where R/G/B stand for red, green and blue colors to paint on edges and an MPG stands for a maximal planar graph. We focus on an extremum non-4-colorable MPG $EP$ in the whole paper. In this second part, we refresh the false proof on $EP$ by Kempe for the Four Color Theorem.  And then using single color tilings or RGB-tilings on $EP$, we offer a renewal point of view through R/G/B Kempe chains to enhance our coloring skill, either in vertex-colorings or in edge-colorings. We discover many fundamental theorems associated with R-/RGB-tilings and 4-colorability; an adventure study on One Piece, which is either an MPG or an $n$-semi-MPG; many if-and-only-if statements for $EP-\{e\}$ by using Type A or Type B $e$-diamond and Kempe chains. This work started on May 31, 2018 and was first announced by the author~\cite{Liu2020} on Jan.\ 22, 2020, when the pandemic just occurred. 
\end{abstract}     

\setcounter{section}{8}
\setcounter{figure}{14}

\maketitle

\section{R/G/B Kempe chains, a renewal point of view} \label{sec:RGBKempeChainDeg5}
Given $EP \in e\mathcal{MPGN}4$, there are at least 12 vertices of degree 5. Let $v_0\in V(EP)$ with $\deg(v_0)=5$. Kempe's classical proof used this fixed vertex $v_0$ and its five neighbors $v_1,v_2,\ldots,v_5$ to perform vertex-color-switching for vertices in two sets, $rC$ and $gC$, where $rC$/$gC$ is a red-/green-connected component linking by 2-4/2-3 edges\footnote{Precisely $rC$ and $gC$ will be denoted by $rC(v_2)$ and $gC(v_5)$ in Subsection~\ref{sec:ReplaceVwE}, because they contain vertex $v_2$ and $v_5$.}. Here we use R/G/B \emph{Kempe chains} to review the old proof by Kempe and discover some important things missed before. 

Two vertices $u,v\in rC$ (a set of red-connected component) must have a red chain (or path) to connect each other. As a red chain or a red-connected component, it could be all 1-3 edges or all 2-4 edge; never mixed. Also it is not just one $u-v$ red chain, but a \emph{cluster} of red chains; however we choose the rightmost chain or the leftmost one to demonstrate the main structure of our target graphs.

To represent $EP$ with $\deg(v_0)=5$ on a flat surface in Figure~\ref{fig:EPcannot}, we show the major part of $EP$ including the five neighbors of $v_0$. Because $EP$ is extremum, $EP-\{v_0\}$ is 4-colorable (Please see Theorem~\ref{RGB1-thm:eMPG4}(b) in Part I of this paper) and $EP-\{v_0\}$ has a 4-coloring function $f$; but $f(v_0)=5$ is inevitable. That means the five neighbors of $v_0$ must using all 4 different colors. Without loss of generality, we draw colors on the neighbors as in  Figure~\ref{fig:EPcannot}. 
   \begin{figure}[h]
   \begin{center}
   \begin{tabular}{ c c }
   \includegraphics[scale=1]{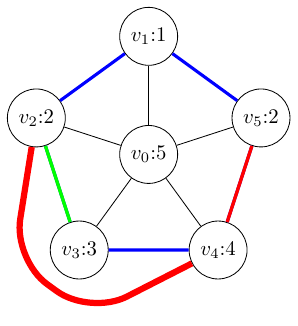}\qquad
   \includegraphics[scale=1]{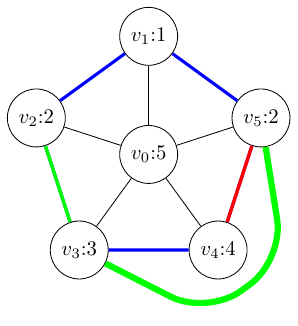}
   \end{tabular}
   \end{center}
   \caption{$EP$ cannot have either $K_r|_{v_2}^{v_4}$ or $K_g|_{v_3}^{v_5}$.} \label{fig:EPcannot}
   \end{figure}

In addition, the two graphs in Figure~\ref{fig:EPcannot} show a 2-4 red path $K_r|_{v_2}^{v_4}$  and a 2-3 green path $K_g|_{v_3}^{v_5}$ respectively. Because of $K_r|_{v_2}^{v_4}$ or $K_g|_{v_3}^{v_5}$, these two graphs are definitely not $EP$. On the left graph, the red path $K_r|_{v_2}^{v_4}$, \underline{blocks any 1-3 red line} from $v_1$ to $v_3$ through $EP-\{v_0\}$. Due to the 1-3 disconnection between $v_1$ and $v_3$, we can perform \emph{vertex-color-switching} on the 1-3 red-connected component containing $v_1$, so that $f(v_1)=1$ turns to be $f(v_1)=3$ without changing the colors of the other four neighbors of $v_0$. Then we can set $f(v_5)=1$ and a 4-coloring function on $EP$ comes out; so $EP$ is not an extremum. Therefore, $EP$ cannot have $K_r|_{v_2}^{v_4}$.  The argument are the same for $K_g|_{v_3}^{v_5}$ in the right graph.

Last paragraph and Figure~\ref{fig:EPcannot} show two forbidden red and green chains for any  extremum $EP$. Then there must be a red \emph{Kempe chain} connecting $v_1$ and $v_3$, denoted by $K_r|_{v_1}^{v_3}$ (or $K_r$ for short), and a green \emph{Kempe chain} connecting $v_1$ and $v_4$, denoted by $K_g|_{v_1}^{v_4}$ (or $K_g$ for short). See the left graph in Figure~\ref{fig:tangle}. 

The existence of $K_r|_{v_1}^{v_3}$ and $K_r|_{v_2}^{v_4}$ is exclusive, i.e., exactly one of them exists, so is the existence of $K_g|_{v_1}^{v_4}$ and $K_r|_{v_3}^{v_5}$. There are many known reasons and one was done by Kempe. Our reason dues to Lemma~\ref{RGB1-thm:InOutCanal}(b) and (c). Especially the red canal system on $EP-\{v_0\}$ creates a non-crossing matching among black (green/blue) edges along the outer facet $\Omega:=v_1$-$v_2$-$\ldots$-$v_5$-$v_1$.   
   \begin{figure}[h]
   \begin{center}
   \includegraphics[scale=0.73]{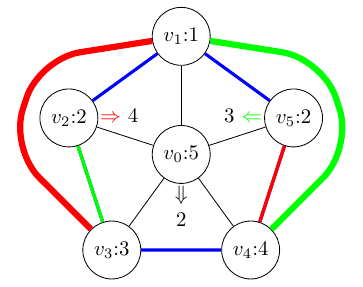}
   \includegraphics[scale=0.64]{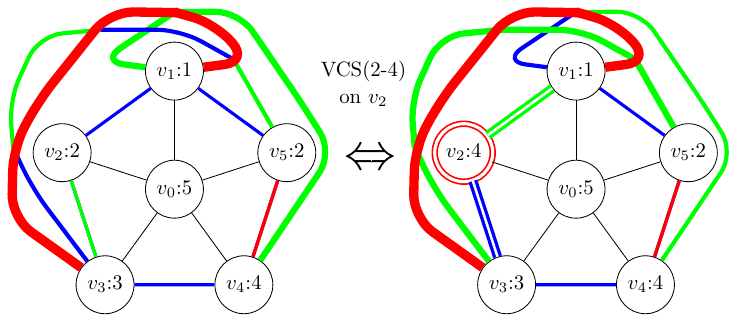}
   \end{center}
   \caption{Kempe's proof and the bug in Kempe's paper.} \label{fig:tangle}
   \end{figure} 
   
Also there are probably many (a \emph{cluster} of) such red chains $K_r|_{v_1}^{v_3}$ and green chains $K_g|_{v_1}^{v_4}$. We shall choose the red chain $K_r$ closest to $v_2$ (the rightmost) and the green chain $K_g$ closest to $v_5$ (the leftmost). This \underline{mandatory} but \underline{temporary} choice concerns the error of Kempe's original proof. Following this choice or ideal, we will claim the \emph{tangling property} that is important to our renewal approach.   

By the same technique which had been applied on the two graphs in Figure~\ref{fig:EPcannot}, Kempe used vertex-color-switching method to get a new coloring as $f(v_2)=4$ and $f(v_5)=3$ without changing the colors of the other three neighbors of $v_0$, namely $f(v_1)=1$, $f(v_3)=3$ and $f(v_4)=4$. Then Kempe finished the proof by setting $f(v_0)=2$. This proof seems perfect by referring the left graph in Figure~\ref{fig:tangle}; otherwise the referees at that time would not pass and then let it publish. Unfortunately, Percy Heawood found the bug in Kempe's paper after 11 years. Briefly \underline{we cannot do} two vertex-color-switching processes w.r.t.\  a red-connect component and a green-connect component \underline{at the same time}. The right two graphs in Figure~\ref{fig:tangle} are what really happens in $EP$; otherwise it is not a real $EP$: Two Kempe chains $K_r|_{v_1}^{v_3}$ and $K_g|_{v_1}^{v_4}$ cross each other shown by the middle graph in Figure~\ref{fig:tangle}. Not only this crossing, once we do the first vertex-color-switching process on the red-connected component containing $v_2$, which associates with the red Kempe chain $K_r|_{v_1}^{v_3}$, to get $f(v_2)=4$, immediately the green Kempe chain $K_g|_{v_1}^{v_4}$ will be destroyed and a new green Kempe chain $K_g|_{v_3}^{v_5}$ will show up; then the second vertex-color-switching process claimed by Kempe cannot fulfill. See the right two graphs in Figure~\ref{fig:tangle}. We use two double-lines and a double-circle to highlight the change on the pentagon.

Symmetrically, provided middle graph in Figure~\ref{fig:tangle}, if we first perform vertex-color-switching process on the green-connected component contain $v_5$, then the red Kempe chain $K_r|_{v_1}^{v_3}$ of the middle graph in Figure~\ref{fig:tangle} will be destroyed and a new red Kempe chain $K_r|_{v_2}^{v_4}$ turns out. Again, the second second vertex-color-switching process claim by Kempe cannot be done. Sorry, we do not offer the two graphs of changing before and after.  

   \begin{remark} 
Percy Heawood used the same idea from Kempe's paper but only perform one vertex-color-switching process to prove Five Color Theorem.
   \end{remark}

\section{The tangling property w.r.t.\ a degree 5 vertex in $EP$} \label{sec:tanglingProperty}
In this section we always set $v_0\in V(EP)$ with $\deg(v_0)=5$. The existence of the dual Kempe chains and the tangling property w.r.t.\ $(EP; v_0)$ are the starting point to transform Kempe's method to our renewal approach. 
  \begin{definition}   \label{def:dualKempeChains}
Let $EP \in e\mathcal{MPGN}4$ and $v_0\in EP$ with $\deg(v_0)=5$. The pair $(K_r, K_g)$ of two crucial chains demonstrated in Figure~\ref{fig:DualKempeChains}, i.e., two solid red and green curves, are called the \emph{dual Kempe chains}  w.r.t.\ $(EP; v_0)$ provided $\deg(v_0)=5$. Precisely we see the dual Kempe chains $(K_r|_{v_1}^{v_3}, K_g|_{v_1}^{v_4})$ w.r.t.\ $(EP; v_0)$. (Please, ignore all dashed lines at this moment.) The subgraph $EP-\{v_0\}$ is a 5-semi-MPG with its pentagon outer facet $\Omega:=v_1$-$v_2$-$\ldots$-$v_5$-$v_1$. By $\Omega$, this $EP$ is partitioned into two regions: $\Sigma$ (inside) and $\Sigma'$ (outside) with  $\Sigma\cap\Sigma'=\Omega$.
   \end{definition}
   
   \begin{definition}    \label{def:TanglingProperty}
Please, continue with the setting of Definition~\ref{def:dualKempeChains} which is demonstrated by Figure~\ref{fig:DualKempeChains} and the right two graphs in in Figure~\ref{fig:tangle}. Now we define the \emph{tangling property} that happens when we perform vertex-color-switching on $v_2$ (or $v_5$) \emph{along} the current $K_r|_{v_1}^{v_3}$ (or $K_g|_{v_1}^{v_4}$). The main code of the tangling property is: After vertex-color-switching, (A) $K_g|_{v_1}^{v_4}$ (or $K_r|_{v_1}^{v_3}$ respectively) will be destroyed and (B) a new Kempe chain $K_g|_{v_2}^{v_3}$ (or $K_r|_{v_2}^{v_4}$) will be created. The property guarantees that new dual Kempe chains $(K_r, K_g)$ w.r.t.\ $(EP; v_0)$ still exist. Please refer to the right two graphs, which are before switching and after, in Figure~\ref{fig:tangle}. 
   \end{definition}
   \begin{figure}[h]
   \begin{center}
   \includegraphics[scale=1]{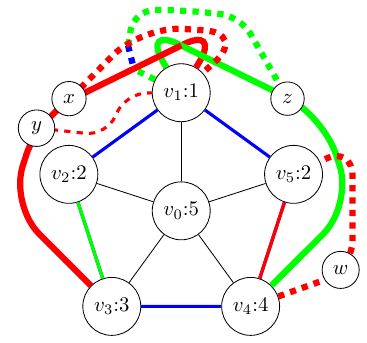}
   \end{center}
   \caption{The dual Kempe chains; ignoring all dashed lines first.} \label{fig:DualKempeChains}
   \end{figure}
  
The general setting for $(EP; v_0)$ and the five neighbors of 
$v_0$ in Figure~\ref{fig:DualKempeChains} is mandatory, where ``general'' means symmetry of vertex colors 1/2/3/4 as well as edge colors R/G/B and any rotation of the pentagon outer facet $\Omega$. Here we list the key properties for $(EP; v_0)$ that we observed in the last section:
   \begin{enumerate}
\item $EP-\{v_0\}$ is 4-colorable; so any 4-coloring function on $EP-\{v_0\}$ must has four different color assign to $v_1,\ldots,v_5$. (By Theorem~\ref{RGB1-thm:eMPG4}(b))
\item Due to (1), the edge coloring on $\Omega$ must be $T_{rgb}|_\Omega:=$ [red-blue-blue-green-blue] or its symmetry. Notice that [red-blue-blue-blue-green] makes $EP$ 4-colorable and it is a contradiction and impossible.  
\item There must exist the dual Kempe chains $(K_r, K_g)$.
\item The tangling property holds w.r.t.\ $(EP; v_0)$ and $(K_r, K_g)$.
\item The most important thing is that $T_{rgb}|_\Omega$ and $(K_r, K_g)$ must match each other.     
   \end{enumerate}
With help of $K_r$ or $K_r$, one can perform \emph{vertex-color-switching} according to Kempe's method. We are going to transfer Kempe's method to our new method: \emph{edge-color-switching}. 
  \begin{definition}   \label{def:Edge-CS}
Given an MPG or semi-MPG $M$ with an RGB-tiling $T_{rgb}=(T_{r},T_{g},T_{b})$ (coexisting triple), the process of \emph{edge-color-switching} on a red canal line $rCL$ of $T_{r}$ (or \emph{along} the left/right canal bank $rCL^l$/$rCL^r$) is to exchange edge-colors green and blue in between $rCL^l$ and $rCL^r$. After this process, we obtain a new and legal RGB-tiling $T'_{rgb}$ without changing $T_{r}$ and they are still coexisting.        
  \end{definition}

Let us use acronyms VCS and ECS to stand for ``vertex-color-switching'' and ``edge-color-switching'' respectively. In some circumstance, one VCS is equivalent to a combination of multiple ECS, and vice versa. We will explain this equivalence behind.

In the last section, we claimed to choose the closest red chain $K_r|_{v_1}^{v_3}$ to $v_2$, and choose the closest green chain $K_g|_{v_1}^{v_5}$ to $v_5$. The ones we choose are drew by two solid red/green lines and they intersect each other. Due to ``closest'', the thin red dashed-line connecting $y$ and $v_1$ does not exist; especially it has no intersection with $K_g|_{v_1}^{v_5}$ and no intersection will force $EP$ 4-colorable by Kempe's proof. We have the second meaning for ``closest'': Once these two closest dual chains intersect each other, any two red and green chains of same end-points shall intersect. Intersection is the minimum requirement to obey the tangling property, especially (A). The detail proof will be offered later. 

The author leave an important question: Besides degree 5, does there any other situation have tangling property?

\subsection{Vertex-color-switching vs edge-color-switching} \label{sec:ReplaceVwE}
Basically Kempe used red-connected component containing $v_2$ to perform VCS. Actually any red Kempe chain from $v_1$ to $v_3$ separate the two \emph{major} red-connected components that contain $v_2$ and $v_4$/$v_5$ respectively. Let us denoted these two components by $rC(v_2)$ and $rC(v_4;v_5)$. Any different $K_r$ from $v_1$ to $v_3$ in its own cluster can be a boundary working zone for VCS/ECS and $K_r$ surrounds $rC(v_2)$ tightly or loosely. For instance, the original $K_r|_{v_1}^{v_3}$ or $K'_r:=v_1$-(dashed red line)-$x$-$y$-(solid red line)-$v_3$ surrounds $rC(v_2)$ tightly or loosely respectively.

   \begin{lemma}  \label{thm:deg5tangling}
Let $(K_r, K_g)$ be any dual Kempe chains w.r.t.\ $(EP; v_0)$ provided $\deg(v_0)=5$. Then $K_r$ and $K_g$ must intersect each other. (We shall ignore their common endpoint $v_1$.)       
   \end{lemma}
   \begin{proof}
Suppose there are $K_r|_{v_1}^{v_3}$ and $K_g|_{v_1}^{v_4}$ without intersection. Because $K_r$ and $K_g$ are boundaries of $rC(v_2)$ and $gC(v_4)$ respectively, no intersection means $V(rC(v_2))\cap V(gC(v_4))=\emptyset$. And then Kempe's proof works; hence $EP \notin e\mathcal{MPGN}4$; this is a contradiction.    
   \end{proof}
\noindent
The non-empty overlapping area $V(rC(v_2))\cap V(gC(v_4))$ is complicate and hard to study, or even hard to draw it due to the tangling property.

Let us focus on $Q:=EP-\{v_0\}$ with R-tiling $T_r(Q)$. By $T_r$, we find three \emph{major} red-connected components: $rC(v_2)$, $rC(v_1;v_3)$ and $rC(v_4;v_5)$; and also two \emph{major} red canal lines $rCL(v_1v_2)$ and $rCL(v_3v_4)$. Please, refer to Figure~\ref{fig:Q} for the definition of these major parts. The pattern on Figure~\ref{fig:Q} for $(Q;T_r)$ is mandatory and offers a new point of view to see four-coloring problems through the method of edge-color-switching. Besides these three major red-connected components, there are also some minor red-connected components $rC_{1i}$, $rC_{2j}$ and $rC_{3k}$ inside $rC(v_2)$, $rC(v_1;v_3)$ and $rC(v_4;v_5)$ respectively.
   \begin{figure}[h]
   \begin{center}
   \includegraphics[scale=1]{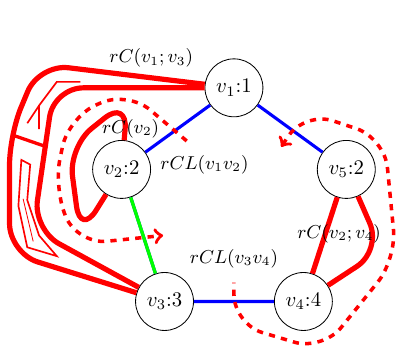}
   \end{center}
   \caption{$rC(\ast)$ and $rCL(\ast)$ w.r.t.\ $(EP-\{v_0\};T_{rgb})$} \label{fig:Q}
   \end{figure}

Furthermore, let us build a new graph call the \emph{red block graph} from $Q:=EP-\{v\}$ and $T_{rgb}$, denoted by $rBG(Q;T_{rgb})$ or $rBG(Q)$ for short, where the block (or vertex) set $V(rBG(Q))$ consists of all red-connected components $rC(\ast)$ and $rC_{ij}$ of $T_r$ and link (or edge) set $V(rBG(Q))$ consists of all red canal lines, such that for each red canal line $rCL(\ast)$ links the two sides of red-connected components that contain $rCL^l(\ast)$ and $rCL^r(\ast)$ respectively.   Please, see Figure~\ref{fig:treeofQ} for example.
   \begin{figure}[h]
   \begin{center}
   \includegraphics[scale=1]{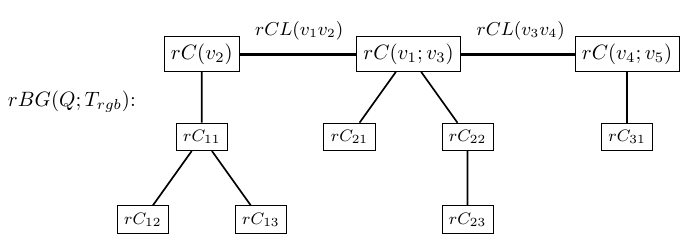}
   \end{center}
   \caption{The red block graph $rBG(Q;T_{rgb})$ obtained from $(Q;T_{rgb})$} \label{fig:treeofQ}
   \end{figure}
In Figure~\ref{fig:Q}, we do draw some details inside $rC(v_1;v_3)$, where $rC_{2j}$ for $j=1,2,3$ are inside $rC(v_1;v_3)$. We have $rC_{21}$ and $rC_{22}$ inside and near by $rC(v_1;v_3)$; also $rC_{23}$ is inside $rC_{22}$. Please, refer to Figure~\ref{fig:treeofQ} for the the consequence of these four blocks shown in $rBG(Q;T_{rgb})$. 

We admit that $rBG(Q;T_{rgb})$ is a tree first. The tree property is very important but we postpone the proof a while. Comparing with the original $rBG(Q;T_{rgb})$ in Figure~\ref{fig:treeofQ}, we use the following two graphs to demonstrate VCS on $rC(v_2)$ and ECS on $rCL(v_1v_2)$ respectively,
   \begin{figure}[h]
   \begin{center}
   \includegraphics[scale=0.87]{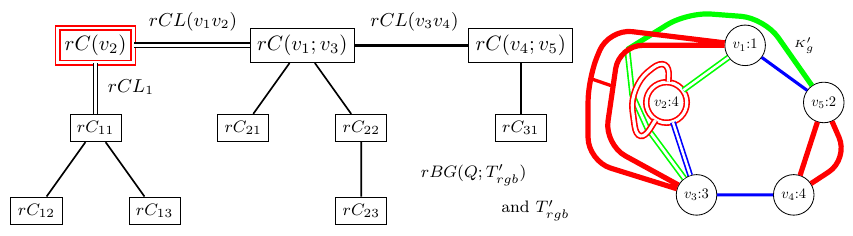}\\
   \includegraphics[scale=0.87]{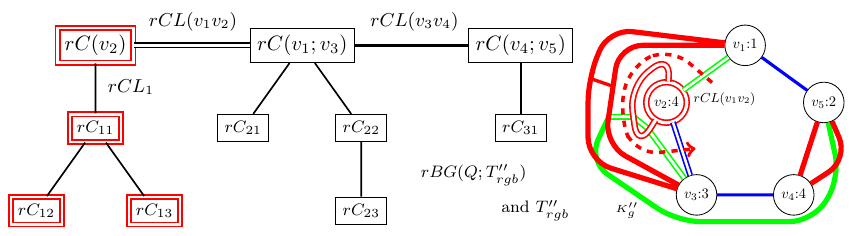}
   \end{center}
   \caption{Top line: VCS on $rC(v_2)$; Bottom line: ECS on $rCL(v_1v_2)$} \label{fig:treeofQ2}
   \end{figure}
where we use doublelines to represent the switching between 1/3  (or 2/4) for vertices in $Q$ as well as switching between green/blue for edges in $Q$. With help from Figure~\ref{fig:treeofQ2}, we have two important observations as follows:
   \begin{itemize}
\item[(V):] When we perform VCS on a single red-connected component $rC$, not only $rC$ has to change but also the links $rCL_i$ that incident to this block $rC$ need to change by ECS. This observation tells us: VCS operation can be replaced by ECS operation.     
\item[(E):] When we perform ECS on a single red canal line $rCL$, not only $rCL$ has to change but also all blocks $rC_j$ that are on one side of $rCL$ need to change by VCS. Notice that a graph without loop have perfect meaning of ``two sides'' of any single link or edge. Therefore, when $rBG(Q)$ is a tree, ECS operation can be replaced by VCS operation.         
   \end{itemize}
This observation is just what we said ``one VCS is equivalent to a combination of multiple ECS, and vice versa'' when $rBG(Q)$ is a tree.

   \begin{lemma}  \label{thm:BGisTree}
Let $Q$  be an MPG or an $n$-semi-MPG (not just $Q:=EP-\{v\}$) with an R-tiling $T_r$. The red block graph $rBG(Q;T_{rgb})$ must be a tree.    
   \end{lemma}
   \begin{proof}
If a graph is bisected by any edge, then it must be a tree. Let $rCL$ be any edge of $rBG(Q;T_{rgb})$. This $rCL$ divide $Q$ into two disconnect regions, because $Q$ is a planar graph and $rCL$ is either a red canal ring or a red canal line which star at and end at a same outer facet.   
   \end{proof}
\noindent
For counterexamples in an $(n_1,n_2)$-semi-MPG, please refer to Figures~\ref{RGB1-fig:counterEX3}(B1), \ref{RGB1-fig:counterEX31}(B2) and~\ref{RGB1-fig:counterEX31}(B3). The observations (V) and (E) tell us: VCS and ECS are exchangeable if $rBG(Q;T_{rgb})$ is a tree. Since we focus on One Piece, we nearly assume the block graphs in our discussion are all trees. In the following discussion and example, we will demonstrate that ECS is much more convenient.

In view of chip-firing games on $rBG(Q;T_{rgb})$ of a tree structure:
   \begin{itemize}
\item[(V):]  A block is selected to chip-fire, then this block and its adjacent link shall switch between singleline and doubleline.        
\item[(E):] A link is selected to chip-fire, then this link and all blocks on one side of this link shall switch between singleline and doubleline. We will use Example~\ref{ex:ECStwoSide} to explain that choosing different side just causes another equivalent RGB-tiling.   
   \end{itemize}

\subsection{Synonym, equivalence and congruence}
\label{sec:SynEquivCong}
The three different red block graphs $rBG(Q;T_{rgb})$, $rBG(Q;T' _{rgb})$ and $rBG(Q;T''_{rgb})$ in the last subsection have no difference on their structure, because they share the same $T_r$. But they do have difference on G-/B-tilings. Comparing with the original $rBG(Q;T_{rgb})$, we use doublelines to indicate the changes by VCS on some $rC_i$ and ECS on some $rCL_j$ for $rBG(Q;T'_{rgb})$ and $rBG(Q;T''_{rgb})$. Even though the two operations given in Figure~\ref{fig:treeofQ2} have different affects on some blocks $rC_i$ and links $rCL_j$, the two results of change are the same in some sense, i.e., they both transform $(EP-\{v_0\},T_{rgb})$ shown as Figure~\ref{fig:treeofQ} to be the two right graphs in Figure~\ref{fig:treeofQ2}. Let us use $T'_{rgb}$ and $(K_r,K'_g)$ to denote the new things created by this VCS operation and also $T''_{rgb}$ and $(K_r,K''_g)$ created by this ECS operation. These two corresponding things are different in details; however both $K'_g$ and $K''_g$ are definitely from $v_3$ to $v_5$. In other words, $gBG(Q;T'_{rgb})$ and $gBG(Q;T''_{rgb})$ are the same in some sense; but they are totally different from the original green block graph $gBG(Q;T_{rgb})$ due to the tangling property.

We are now discussing about ``the same in some sense'' or ``the difference in certain levels'' that involves three general definitions. Let $M$ be an MPG or a semi-MPG, and $\mathcal{RGBT}(M)$ be the set of all RGB-tilings on $M$.
   \begin{itemize}
\item Synonym: Any $T_{rgb}\in\mathcal{RGBT}(M)$ has six \emph{synonyms}, including itself, by interchanging among R/G/B over whole graph $M$.  This relation of synonym, denoted by $\overset{\text{\tiny syn}}{=}$, is the most basic idea and it is too trivial to mention most of the time. In addition, any kind of synonyms caused by permutations of R/G/B shall be also denoted by $\overset{\text{\tiny syn}}{=}$. We also use $\langle T_{rgb}\rangle$ to denote the set of six synonyms of $T_{rgb}$. But sometimes we will even skip $\langle \cdot\rangle$. 
\item Equivalence: First, we shall accept the fundamental base on synonyms. The most important parts of $(EP;v_0)$ with $\deg(v_0)=5$ are $(K_r,K_g)$ and $T_{rgb}|_\Omega$. This two parts are the major \emph{skeleton} of any kind of $T_{rgb}(EP;v_0)$. In general, any two $T^A_{rgb}, T^B_{rgb}\in\mathcal{RGBT}(M)$ are \emph{equivalent}, denoted by $T^A_{rgb}\equiv T^B_{rgb}$, if they share the same skeleton such as the graph in Figure~\ref{fig:DualKempeChains} (ignoring all dashed lines), i.e., the same sketch for $(K_r,K_g)$ and $T_{rgb}|_\Omega$. The two right graphs in Figure~\ref{fig:treeofQ2} do have same skeleton and they provide another example: $T'_{rgb}\equiv T''_{rgb}$. {\bf It is important that this equivalence relation shall involve with  a given $\Omega$ and Kempe chains in $\Sigma'$. Different $\Omega$'s will establish different equivalence relations.} We will talk about this ``difference'' then. We will use  $[T_{rgb}]$ to denote the equivalence class that $\langle T_{rgb}\rangle$ belongs to. There is supplemental definition of equivalence given behind in Remark~\ref{re:EquivInSigma}.  
\item Congruence: First, this new relation bases on accepting $\equiv$ or $\overset{\text{\tiny syn}}{=}$. We have seen an example: $T_{rgb}$ and $T'_{rgb}$ (or $T_{rgb}$ and $T''_{rgb}$) in the last paragraph and also in the last subsection. The congruence relationship has its operational definition: In the working domain $\mathcal{RGBT}(M)$, two RGB-tilings $T^A_{rgb}$ and $T^B_{rgb}$ are \emph{congruent}, denoted by $T^A_{rgb} \cong T^B_{rgb}$, if $T^B_{rgb}$ can be obtained from $T^A_{rgb}$ by performing a sequence of VCS's and ECS's.  How to make VCS and ECS executable and closed in $\mathcal{RGBT}(M)$? For instance, we need to require any $rC$ (or $gC$, $bC$) having no odd-cycles to perform VCS. Also, after operation the result should still be an element in the working domain $\mathcal{RGBT}(M)$. We need to set up a stronger requirement on $\mathcal{RGBT}(M)$ and choosing a proper $M$ is what we need to do.
      \begin{itemize}
\item[(1)] Let $M$ be \underline{One Piece} which is an MPG or an $n$-semi-MPG. This is a good choice due to Theorem~\ref{RGB1-thm:4RGBtiling} which is our First Fundamental Theorem v1.
\item[(2)] Let $M$ be an MPG or an semi-MPG. This choice due to Theorem~\ref{RGB1-thm:4RGBtilingGeneralization} which is our First Fundamental Theorem v2. For this setting, we also need to restrict a new domain set $\mathcal{RGBT}^+(M)$ that consists of all $T_{rgb}$ such that along every $n_i$-gon outer facet the numbers of red, green and blue edges are all even if $n_i$ is even, and all odd if $n_i$ is odd (Theorem~\ref{RGB1-thm:4RGBtilingGeneralization}(c)).
\item[(3)] Let $M=EP-\{\ast\}$, where $\{\ast\}$ is dynamic and consists of several edges that are variable. Precisely we have a fixed cycle $\Omega$ in $EP$, and $\Sigma$ and $\Sigma'$ are two regions of $EP$ partitioned by $\Omega$ with $\Sigma \cap \Sigma'=\Omega$. In addition, the variable edge set $\{\ast\}$ is always inside $\Sigma$. Since $\Sigma'$ is One Piece, all rules shall follow item (1); Since $\Sigma$ has multiple outer facets, all rules shall follow item (2); Of course some special rules due to the combination of (1) and (2).  We will talk about it then.          
      \end{itemize}
   \end{itemize}

   \begin{remark}\label{re:Question} (Very important)
The foundation of synonym relation can be any 4-colorable graph $M$ and $\mathcal{RGBT}(M)$. The foundation of equivalence relation need to establish a certain skeleton; here we use $(K_r,K_g)$ on planar graph $M=EP-\{v_0\}$ and $T_{rgb}(M)|_\Omega$. The foundation of congruence relation is a 4-colorable \underline{planar} graph $M$ and $\mathcal{RGBT}(M)$; here we still use $M=EP-\{v_0\}$. Both synonym and congruence relations are defined by certain ECS operations, but equivalence relation is defined by the way (skeleton) we draw for $(K_r,K_g)$ and $T_{rgb}|_\Omega$. The crucial question come out as follows:
   \begin{center}
   \begin{tabular}{c c c}
$T^A_{rgb}$  & $\equiv$ & $T^B_{rgb}$\\
{\Large $\Downarrow$}, $\cong$ & {\qquad corresponding ECS \qquad} & {\Large $\Downarrow$}, $\cong$\\  
$T^{A'}_{rgb}$  &  $\equiv$\rlap{\Large ?} & $T^{B'}_{rgb}$   
   \end{tabular}
   \end{center} 
The answer is yes and we use the next subsection to persuade the reader. Notice that the corresponding ECS's has three different groups: (1) ECS on $rCL(v_1,v_2)$, (2) ECS on $rCL(v_3,v_4)$, and (3) any ECS on $rCL_k$ which is inside $rC(v_2)$, $rC(v_1;v_3)$ or $rC(v_4;v_5)$. It is interesting that after performing (3) the new RGB-tiling is equivalent to old one, not just congruent. 
   \end{remark}

\subsection{Let us learn ECS by examples} \label{sec:LearnECS} The right two graphs in Figure~\ref{fig:tangle} are closer to reality, but it is not easy (actually no way) to draw the real graph for any dual Kempe chains $(K_r,K_g)$ w.r.t.\ $(EP; v_0)$. In the rest of this paper, we will simply draw $(K_r,K_g)$ like the left graph in Figure~\ref{fig:tangle} to show the red-/green-connected property. However, we always pretend that the real $K_r$ and $K_g$ behind the graph do intersect each other and the tangling property always working.  

For the following two examples, we star with $(Q;T_{rgb})$ in Figure~\ref{fig:Q} as well as $rBG(Q;T_{rgb})$ in Figure~\ref{fig:treeofQ} and perform many different ways of ECS.
   \begin{example} \label{ex:ECStwoSide}
Observation (E) shows: performing ECS on $rCL(v_1v_2)$ shall simultaneously perform VCS on all blocks $rC_j$ on one side of $rCL(v_1v_2)$. In Figure~\ref{fig:treeofQ2} we choose the left side to simultaneously perform VCS. How about we choose the right side? No problem, the result $T^1_{rgb}$ shown in Figure~\ref{fig:learnECS} tell us $T^1_{rgb} \equiv T'_{rgb}\equiv T''_{rgb}$.   
   \begin{figure}[h]
   \begin{center}
   \includegraphics[scale=0.87]{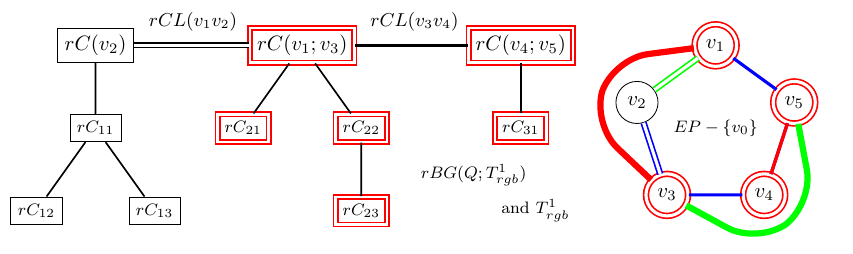}
   \end{center}
   \caption{ECS on $rCL(v_1v_2)$ again, but affecting the other side} \label{fig:learnECS}
   \end{figure}   
   \end{example}

We said that the blocks $rC(v_2)$, $rC(v_1;v_3)$, $rC(v_4;v_5)$, and the links $rCL(v_1v_2)$, $rCL(v_3v_4)$ are five major parts of $rBG(EP-\{v\})$ and they forms a line of length 2. Besides these major fives, $rC_{ij}$ and $rCL_k$ are inside  $rC(v_2)$, $rC(v_1;v_3)$ or $rC(v_4;v_5)$. 
   \begin{lemma} \label{thm:rCijrCLk}
Given $rBG(E-\{v_0\};T_{rgb})$, if we perform any combination of VCS on $rC_{ij}$, together with any combination of ECS on $rCL_k$ and obtain a new $T'_{rgb}$ on $E-\{v_0\}$, then the original $T_{rgb}|_\Omega$ and $T'_{rgb}|_\Omega$ are the same and then $(K'_r,K'_g)$ and $(K'_r,K'_g)$ of the same kind. Precisely $T_{rgb}\equiv T'_{rgb}$.
   \end{lemma}

This lemma is the main reason that we only draw $K_r|_{v_1}^{v_3}$ for red-connectivity, instead of $rC(v_1;v_3)$ as a component. We also relax the previous \underline{mandatory} but \underline{temporary} choice, because now we know what we really is red-/green-connectivity and all kind of $(K_r, K_g)$'s follow the tangling property for $(EP;v_0)$.  

   \begin{figure}[h]
   \begin{center}
   \includegraphics[scale=0.87]{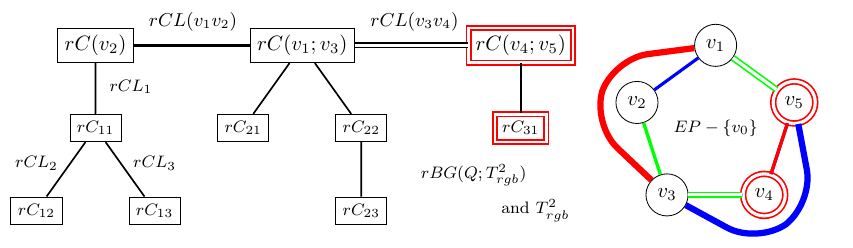}
   \includegraphics[scale=0.87]{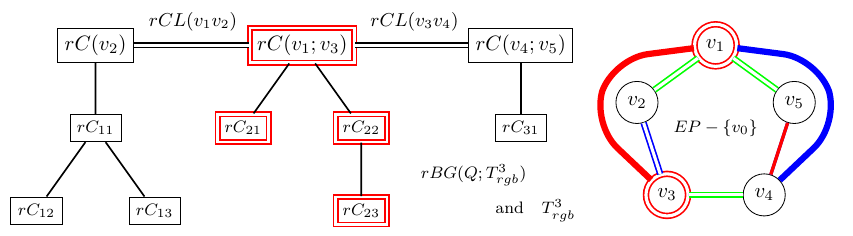}
   \end{center}
   \caption{The red block graphs for VCS on $rC(v_2)$ and ECS on $rCL(v_1v_2)$} \label{fig:learnECS2}
   \end{figure}   
   \begin{example}
We are curious about ECS on $rCL(v_1v_2)$ or $rCL(v_3v_4)$, or on them together. Performing ECS on $rCL(v_1v_2)$ is given in Subsection~\ref{sec:ReplaceVwE} and we obtain 
$T''_{rgb}$ in Figure~\ref{fig:treeofQ2}. There are  two demonstrations in Figure~\ref{fig:learnECS2}, where we perform ECS on $rCL(v_3v_4)$ and perform ECS on both $rCL(v_1v_2)$ and $rCL(v_3v_4)$. The second operation is equivalent to perform VCS on $rC(v_1;v_3)$ by Lemma~\ref{thm:rCijrCLk}. Clearly, $T^2_{rgb}\equiv  T^1_{rgb} \equiv T'_{rgb}\equiv T''_{rgb}$, even though we have $v_3$ and $v_5$ blue-connected rather than green-connected. (Please refer to synonym in the last subsection.)
As for the second operation, we obtain $T^3_{rgb}$ and clearly $T^3_{rgb}\equiv  T_{rgb}$. 
    \end{example}

   \begin{example}
Notice that $T^3_{rgb}$ and $T_{rgb}$ in the last example are not synonyms, because the real synonym of $T_{rgb}$ with $T_r$ fixed need to perform ECS on all $rCL(\ast)$ and $rCL_\ast$. In Figure~\ref{fig:learnECS3} we do offer $T^4$ to be such a synonym of $T_{rgb}$.
   \begin{figure}[h]
   \begin{center}
   \includegraphics[scale=0.87]{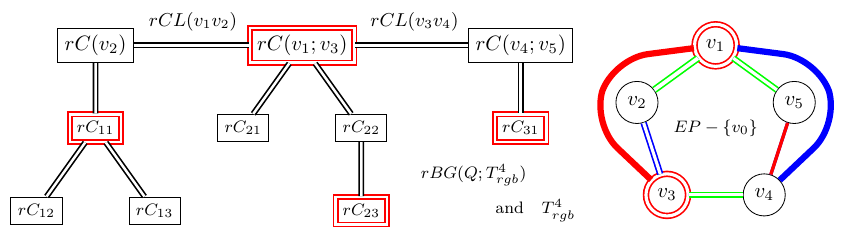}
   \end{center}
   \caption{$T^4$: a synonym of $T_{rgb}$} \label{fig:learnECS3}
   \end{figure}   
   \end{example}

With these examples, we can make a conclusion on any provided RGB-tiling $T_{rgb}$ on  $EP-\{v_0\}$ with $\deg(v_0)=5$ and $rBG(EP-\{v_0\};T_{rgb})$ shown in Figure~\ref{fig:Q} as follows:  
   \begin{enumerate}
\item Since VCS and ECS can be substituted by each other, we could only focus on ECS. If we fixed a $T_{r}$ without red odd-cycles, then there are $2^N$ different coexisting RGB-tiling induced by this R-tiling, where $N$ is the total number of red canal lines $rCL_i$ including both rings and paths. For $EP-\{v_0\}$, we have $N\ge 2$.
\item Among these $2^N$ different coexisting RGB-tilings w.r.t.\ our fixed $T_{r}$, which is generated by the original $T_{rgb}$, we are interested in congruence classes. For congruence ``$\cong$'', if we fixed $T_{r}$, then $T_{rgb}$ only has the other congruent $T'_{rgb}$. For all we met such as $T''_{rgb}, T^1_{rgb}, \ldots, T^4_{rgb}$, they are either synonyms of $T_{rgb}$ or $T'_{rgb}$ or in equivalence ``$\equiv$''.
\item Only performing ECS on $rCL(v_1v_2)$ or $rCL(v_3v_4)$, we can exchange between $[T_{rgb}]$ and $[T'_{rgb}]$. However, performing ECS on both $rCL{v_1v_2}$ and $rCL{v_3v_4}$ exchanges nothing between $[T_{rgb}]$ and $[T'_{rgb}]$. This is why we only have two congruence classes if $T_{r}$ is fixed. This result provides the final answer for Remark~\ref{re:Question}.
\item Provided RGB-tiling $T_{rgb}$, we can also draw $gBG(EP-\{v_0\};T_{rgb})$, which is a symmetric graph of $rBG(EP-\{v_0\};T_{rgb})$ in Figures~\ref{fig:DualKempeChains} and~\ref{fig:Q}. So (1), (2)and (3) hold for green version.
\item Provided Figures~\ref{fig:DualKempeChains}, there is no corresponding blue version, because the edge coloring of $T_{rgb}|_\Omega$ shows blue is unique w.r.t.\ red and green. 
   \end{enumerate}

\subsection{Our next step: $EP-\{e\}$ vs $EP-\{v_0\}$} 
Now we use three graphs in Figure~\ref{fig:renew} to extend the idea of R/G/B Kempe chains. These three graphs are special enough to demonstrate the benefit obtained from the new concept using RGB-tilings.   
   \begin{figure}[h]
   \begin{center}
   \begin{tabular}{ c c c}
   \hspace*{-10pt}
   \includegraphics[scale=.91]{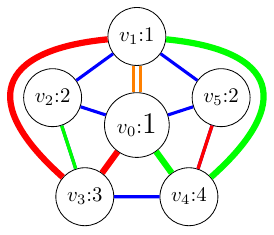}\
   \hspace*{-5pt}
   \includegraphics[scale=.91]{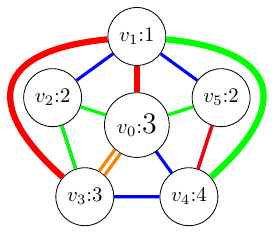}\
   \hspace*{-5pt}
   \includegraphics[scale=.91]{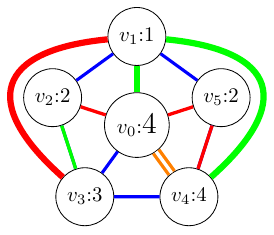}
      \end{tabular}
   \end{center}
   \vspace*{-10pt}
   \caption{Three renewal graphs for R/G/B Kempe chains and $(EP; v_0)$} \label{fig:renew}
   \end{figure}

The original Kempe's point of view focus on  $EP-\{v_0\}$ which is a $5$-semi-MPG with a pentagon outer facet $\Omega:= v_1$-$v_2$-$v_3$-$v_4$-$v_5$-$v_1$ showed as the right graph in Figure~\ref{fig:tangle}. By $\Omega$, this $EP$ is partitioned into two regions: $\Sigma$ (inside) and $\Sigma'$ (outside) with  $\Sigma\cap\Sigma'=\Omega$. Both $\Sigma$ and $\Sigma'$ are 5-semi-MPG's. By the previous general setting Co[$v_1$:1, $v_2$:2, $v_3$:3, $v_4$:4, $v_5$:2] (Co and $f$ are the same thing), we still have the final vertex $v_0$ to color. If we follow the rule of map-coloring, then we must color $v_0$ by the unwelcome color 5. However, this time we choose to obey the rule of only four colors by ignoring a particular edge inside $\Sigma$, while everything in $\Sigma'$ is unchanged. We set the the first graph in Figure~\ref{fig:renew} with Co[$v_0$:1] and then obtain an RGB-tiling on $EP-\{v_0v_1\}$, where the yellow double-line\footnote{This double-line is actually orange color because yellow color in not easy to see for publications.}, namely $v_0v_1$, is the \emph{abandoned edge} at this moment. Notice that the four edges surrounding $v_0v_1$ are all blue, then we name this $v_0v_1$-diamond \emph{Type A}.  Here we demonstrate a new way to realize the dual Kempe chains, namely $(K_r|_{v_0}^{v_1}, K_g|_{v_0}^{v_1})$, which are a little bit longer than the corresponding pairs described in  Definition~\ref{def:dualKempeChains}. By assigning red, green or blue color to that yellow double-line, we will create at least an odd-cycle of the same color, namely $K_r\cup \{v_0v_1\}$, $K_g\cup \{v_0v_1\}$ or two triangles of blue color. Triangles are trivial odd-cycles, so we ignore them most of time and only focus on non-trivial odd-cycles. 

The rest two graphs in Figure~\ref{fig:renew} are obtained by treating $v_0v_3$ and $v_0v_4$ as \emph{abandoned edges} respectively. A little bit different is that the four edges surrounding $v_0v_3$ (or $v_0v_4$) are two blue and two green (red respectively).  We name this kind of $v_0v_3$-diamond as well as $v_0v_4$-diamond \emph{Type B}. The crucial concept is that the middle graph has one Kempe chain $K_r|_{v_0}^{v_3}$ and the right graph has one Kempe chain $K_g|_{v_0}^{v_4}$.

   \begin{definition}
Let $EP \in e\mathcal{MPGN}4$ with $\deg(v_0)=5$. Referring to the first graph in Figure~\ref{fig:renew}, we define $(K_r|_{v_0}^{v_1}, K_g|_{v_0}^{v_1})$ to be the \emph{dual Kempe chains} w.r.t.\ $(EP; v_0v_1)$ in \emph{Type A}. Without change edge coloring in $\Sigma'$, referring to the second graph in Figure~\ref{fig:renew}, we define $K_r|_{v_0}^{v_3}$ to be the \emph{Kempe chain} w.r.t.\ $(EP; v_0v_3)$ in \emph{Type B}.  For each of the three graphs, we call the diamond with yellow double-line the \emph{$e$-diamond} in $EP$.
   \end{definition}

   \begin{remark}
The surrounding four edges of $e$-diamond being same color is the main characteristic of Type A. Type B has two different colors for the surround four edges of $e$-diamond: two edges in the north-$\wedge$ and the other two in the south-$\vee$ that have same color. 
   \end{remark}
   
   \begin{remark} \label{re:EquivInSigma}
(Important) Because 4-semi-MPG $Q:=EP-\{e\}$ is 4-colorable for any $e\in E(EP)$, there exists at least an RGB-tiling on $Q$. By Figure~\ref{fig:renew}, we see one Type A and two Type B RGB-tilings on $Q:=EP-\{\ast\}$, where $\{\ast\}$ consists of only one  edge as variable $e$. We see $e$ can be $v_0v_1$, $v_0v_3$ and $v_0v_4$. These three graphs coexist; so, are they three synonyms? Are they equivalent? Definitely they do not involve congruence. For a fixed $T_{rgb}|\Omega$ and a fixed edge-color-skeleton $(K_r,K_g)$, we shall say these three graphs equivalent. Now we shall use this supplement to claim our standard operating procedure to build on relation of equivalence: 
   \begin{enumerate} 
\item Let $EP \in e\mathcal{MPGN}4$, $\Omega$ be a cycle in $EP$, and $\Sigma$, $\Sigma'$ defined as usual. First we pick any $e_0$ inside $E(\Sigma)$ as well as an $e_0$-diamond, and then develop a Type A RGB-tiling $T_{rgb}$ on $EP-\{e_0\}$.  It is better that one of $e_0$'s two end vertices is degree 5 or 6. This is exactly the left graph in Figure~\ref{fig:renew}. Now we have at least a dual Kempe chains $(K_r,K_g)$ and also $T_{rgb}|_\Omega$ is now fixed.
\item According to this fixed $T_{rgb}|_\Omega$, we can develop new RGB-tilings on $T^i_{rgb}$ on $\Sigma-\{\ast\}_i$ for $\{\ast\}_i$ consisting of a single edge $e_i$ or even more edges from $E(\Sigma)$. This is exactly the right two graphs in Figure~\ref{fig:renew} and  Figure~\ref{fig:renew2} behind.
\item Notice that we must have $T^i_{rgb}|_\Omega = T_{rgb}|_\Omega$. For $T^i_{rgb}(\Sigma-\{\ast\}_i)$, we might develop some new R/G/B Kempe chain in $\Sigma'$ who should not contradict with each other, especially with the original  $(K_r,K_g)$. Sorry! No new R/G/B Kempe chains appear in Figures~\ref{fig:renew} and~\ref{fig:renew2}.
\item All together, we have the \emph{skeleton}: $T_{rgb}|_\Omega$, all feasible R/G/B Kempe chains in $\Sigma'$ and $T_{rgb}|_\Sigma$, $T^i_{rgb}|_\Sigma$ to form an equivalent class, denote by $[T_{rgb}]$
   \end{enumerate}   
   \end{remark}
\noindent
Why we need equivalence relation? Because all properties or proofs in this paper depend on the skeleton of $T_{rgb}$. If it is right for $T_{rgb}$, then it is right for $[T_{rgb}]$. 
     
   \begin{remark}
Without change edge coloring in $\Sigma'$, let us set Co[$v_0$:2]. 
Please see Figure~\ref{fig:renew2}. This way will creates two abandoned edges, namely $v_0v_2$ and $v_0v_5$. Notice that $v_0v_2$- and $v_0v_5$-diamonds are both Type B. 
   \begin{figure}[h]
   \begin{center}
   \includegraphics[scale=.91]{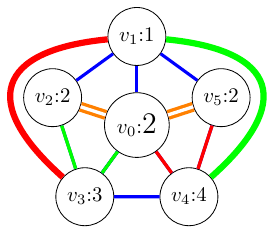}
   \end{center}
   \caption{Two yellow double-lines in $\Sigma$} \label{fig:renew2}
   \end{figure}
However, the surround edges are involving 3 different edge-colors and this time no non-trivial odd-cycles come out by assigning both abandoned edges either red, green or blue, or even mixed with $v_0v_2$ red and $v_0v_5$ green. Thus, we obtain no benefit by setting Co[$v_0$:2] in this case. However, it does not mean that having more abandoned edges at the same time is worthless. What we really care about is any Kempe chain that crosses odd number of abandoned edges.  
   \end{remark}

   \begin{remark}
When we have a Type A $e$-diamond as the left graph in Figure~\ref{fig:renew}, we might want to replace the yellow double-line by red color (or green).  We actually treat the provided RGB-tiling on $Q:=EP-\{e\}$ as an R-tiling $T_r(EP-\{e\})$ which definitely has no odd-cycles. At this moment, green and blue colors are treated as black. 
Replacing yellow by red color will create a new red odd-cycle, because 
now $T_r$ is well defined on $EP$ as an MPG, and then let us refer to Theorem~\ref{RGB1-thm:4RGBtiling}(c).  As for this Type A $e$-diamond, we would not replace the yellow double-line by blue. Even though doing this will reach two trivial blue triangles, they reveal no extra information.  As for the middle (right) graph in Figure~\ref{fig:renew}, we have a Type B $e$-diamond; this time we might want to replace the yellow double-line only by red color (or only by green).
   \end{remark}   
   
The original Kempe chains is w.r.t.\ $(EP; v)$  for a vertex $v$ with $\deg(v)=5$ and our renewal Kempe chains is w.r.t.\ $(EP; e)$ for any edge $e$ in $EP$, while the two end vertices of $e$ need no extra requirement due to Theorem~\ref{RGB1-thm:eMPG4}(b). This subsection or this whole section has paid attention on the connection of Kempe's method and our renewal way. In the next section and in the rest of our study, we will exam more detail and give more properties about R/G/B Kempe chains as well as Type A and Type B $e$-diamonds. 

\section{$e$-diamond everywhere in $EP$} \label{sec:ediamond}

In this section we investigate a general $e$-diamond with any fixed $e\in E(EP)$. Theorem~\ref{RGB1-thm:eMPG4} and Theorem~\ref{RGB1-thm:4RGBtiling} are the top guidelines of this section. As the author reviewed and re-wrote this article $n$ times, the tune of ``Everybody wants to rule the world'' by Tears for Fears, a pop rock band from England, was resonating. Yes, the main theme of this section is ``Every $e$-diamond can rule its world: $EP$.''
   \begin{figure}[h]
   \begin{center}
   \begin{tabular}{ c c c }
   \includegraphics[scale=1.2]{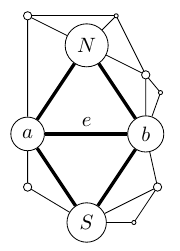}\hspace{15pt}
   \includegraphics[scale=1.2]{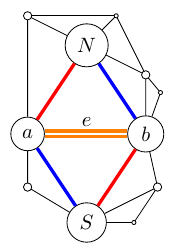}\hspace{15pt}
   \includegraphics[scale=1.2]{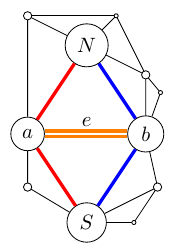}
   \end{tabular}
   \end{center}
   \caption{$e$-diamond, and RGB-tilings of Types C, D on $EP-\{e\}$} \label{fig:impossible}
   \end{figure}

Let us denote the four vertices surrounding the $e$-diamond by $a,b,N,S$ and $e:=ab$. Around this $e$-diamond, the main structure of $EP$ look like the first graph in Figure~\ref{fig:impossible}. Now we try to arrange an RGB-tiling on $Q:=EP-\{e\}$ which definitely exists by Theorem~\ref{RGB1-thm:eMPG4}(b). According to Lemma~\ref{RGB1-thm:evenoddRGB}(b), an RGB-tiling on this 4-semi-MPG $Q$ shall present only one edge-color or two different edge-colors in pairs along the outer facet $\Omega:=N$-$a$-$S$-$b$-$N$. By symmetry or synonym relation, it does matter which one or which two colors are presented. 

   \begin{remark}
Every claim, property or theorem must consider all synonyms behind, i.e., red, green and blue are symmetric and exchangeable. When it comes to synonym relation, we shall also exam the new equivalence relation for this new and general situation. Here we have $\Omega:=N$-$a$-$S$-$b$-$N$ and $\Sigma$ is exactly the $e$-diamond. Please, refer to Remark~\ref{re:EquivInSigma} for more details about building up equivalence relation.
   \end{remark}

First things first, we exclude the two types of RGB-tilings shown as the right two graphs in Figure~\ref{fig:impossible} from $\mathcal{RGBT}(EP-\{e\})$, because they are impossible for $EP$ as an extremum. We simply assign green color to replace the yellow double-line, then we get an RGB-tiling on $EP$.  Assigning $e$ green color causes no green odd-cycle. The reason comes from Lemma~\ref{RGB1-thm:evenoddRGB}(b) applying on this RGB-tiling for 4-semi-MPG $Q$ (not for $EP$). Suppose there is a green path $P_g|_{a}^{b}$. This path together with the 2-path $a$-$N$-$b$ form a cycle in $Q$. The numbers of red, green and blue edges along this cycle are all odd; so the length of $P_g|_{a}^{b}$ is odd and $P_g|_{a}^{b} \cup \{e\}$ is an even-cycle.  Actually, we have another simple way to prove it. Just consider these two graphs with $e$ colored by green as R-tilings on $EP$ (not just for $Q$) because the two red edges together with their two red-triangles (red half-tiles) perfectly share $e$-diamond, and then follow Theorem~\ref{RGB1-thm:4RGBtiling}(a) and (c).

After ruling out the above two types, there are the rest two types of RGB-tilings for 4-semi-MPG $Q$ remained.  We call the two remained ones by \emph{Types A and B} (see Figure~\ref{fig:AtypeBtype00}), and we call the ones ruled out by \emph{Types C and D} (see Figure~\ref{fig:impossible}). All these Types have their own synonyms and equivalence classes; while the four graphs in the two figures are just representations. Now we shall investigate Types A and B.
   \begin{figure}[h]      
   \begin{center}
   \includegraphics[scale=1.2]{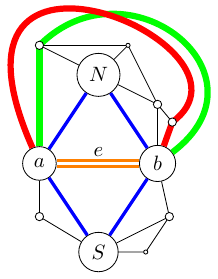}
   \hspace*{20pt}
   \includegraphics[scale=1.2]{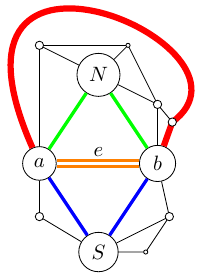}
   \end{center}
   \caption{Type A and Type B RGB-tilings for $EP-\{e\}$}  \label{fig:AtypeBtype00} 
   \end{figure}      

Is there any blue non-trivial odd-cycle? We don't know and most of time we don't need to know. 
In this Type A, the red path and the green path of even length are so called the \emph{dual Kempe chains} $(K_r|_a^b, K_g|_a^b)$ w.r.t.\ $(EP;e)$. Do $(K_r|_a^b, K_g|_a^b)$ have tangling property? They do if $\deg(a)=5$ or $\deg(b)=5$.  We are unsure about this question if  $\deg(a)>5$ and $\deg(b)>5$.  As a representative of Type B, the right graph has only one \emph{Kempe chain} $K_r|_a^b$ guaranteed.

   \begin{theorem}[The primitive Theorem: $e$-iamond of Type A or Type B] \label{thm:EPediamond}
Given $EP\in e\mathcal{MPGN}$, each of the following properties is a necessary condition for $(EP;e)$, where $e:=ab$ be any edge in $EP$ and $Q:= EP-\{e\}$.
   \begin{enumerate}
\item[(a)] All RGB-tilings\footnote{Theorem~\ref{RGB1-thm:eMPG4}(b) guarantees this set non-empty.} on 4-semi-MPG $Q$ can be sorted into two types (or equivalence classes): Type A and Type B shown as in Figure~\ref{fig:AtypeBtype00}. 
\item[(b)] The main characteristics of Type A including: (b1) All four edges surrounding $e$ are the same color, say blue. (b2) There are the dual Kempe chains $(K_r|_a^b, K_g|_a^b)$ w.r.t.\ $(EP;e)$. (b3) The lengths of $K_r|_a^b$ and $K_g|_a^b$ are both even.
\item[(c)] The main characteristics of Type B including: (c1) The four edges surrounding $e$ have two colors, say green and blue, with same color on the north-$\wedge$ and the south-$\vee$ of $e$. (c2) There is a single Kempe chain $K_r|_a^b$  w.r.t.\ $(EP;e)$. (c3) The length of $K_r|_a^b$ is even.
   \end{enumerate}
   \end{theorem}

   \begin{proof}
(a): Since $Q:=EP-\{e\}$ is 4-colorable and then an RGB-tiling $T_{rgb}(Q)$ must exist. We have already rule out Types C and D; so only Types A and B remained under synonym relation. Also the claims (b1) and (c1) are true.   

(c2) and (c3): For Type B, we can replace the yellow double with a red edge and then obtain an extended R-tiling on whole $EP$. Now this new R-tiling cannot induce a 4-coloring function on $EP$, so there must a red odd-cycle passing $e$. Therefore, (c2) and (c3) are true.   

(b2) and (b3): We have the same way to these two. Additionally, we can replace the yellow double with a green edge. 
   \end{proof}

   \begin{remark} \label{re:KnotAsSimpleAs}
Let us refer to the two graphs in Figure~\ref{fig:AtypeBtype00}. A Kempe chain $K_\ast$ is not as simple as a single path. It is possible that $K_r|_\alpha^\beta$ represents a bunch of red paths from $\alpha$ to $\beta$ and then many red canal rings $rCL$ lay inside 
$K_r|_\alpha^\beta$. For instance, Both $K_r|_a^b$ in Type A and Type B graphs actually represent a red-connected component that contains both vertices $a$ and $b$. In view of components, we shall denote it by $rC(a;b)$, and there are also components $rC(N)$, $rC(S)$ and $rC(S)_{ij}$. The two major red canal lines are $rCL(aN)$ and $rCL(aS)$. All these blocks and links will make up the red block graph. Of course, there is the green block graph for Type A, but not for Type B according to the two graphs in Figure~\ref{fig:AtypeBtype00}.         
   \end{remark}

   \begin{theorem} \label{thm:EPediamond2}
Given $(EP;e)$ as in Theorem~\ref{thm:EPediamond}, both RGB-tilings of Types A and B exist. Furthermore, any Type A RGB-tiling is congruent to a Type B one, and vice versa. Therefore, 
  $$
\#\{T_{rgb}(EP-\{e\})\text{ of Type A}\}\ =\ \#\{T_{rgb}(EP-\{e\})\text{ of Type B}\}.
  $$
   \end{theorem}
   \begin{proof}
To explain this, let us make the two RGB-tilings a little bit precise as following two graphs.   
   \begin{figure}[h] 
   \begin{center}
   \includegraphics[scale=1.2]{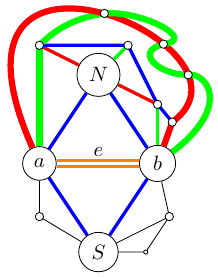} \quad \quad
   \includegraphics[scale=1.2]{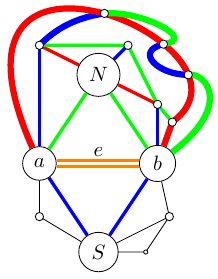}
   \end{center}
   \caption{Congruent partnership between Type A and Type B} \label{fig:AtypeBtype}
   \end{figure}
We start with Type A which is the left graph in Figure~\ref{fig:AtypeBtype}.  Without loss of generality, let us focus on the red Kempe chain $K_r$ from $a$ to $b$. The bounded region enclosed by $K_r\cup e$ is a good place to perform ECS between green and blue.   Then we obtain a Type B shown as the right graph. Notice that the original green Kempe chain from $a$ to $b$ of the left graph is now destroyed. 

The last paragraph is just one direction. To prove the other direction, we cannot use the right graph in Figure~\ref{fig:AtypeBtype} who has an additional green-blue path crossing red path $K_r|_a^b$ in a particular way. The initial Type B has no information about this green-blue path. The correct way is to use the right graph in Figure~\ref{fig:AtypeBtype00} which is 
the general Type B for $EP-\{e\}$ and it only has one Kempe chain $K_r|_a^b$. But the proving process is still reversed: We perform ECS in the a bounded region enclosed by $K_r\cup e$ for this right graph, and then the green color of the edges $aN$ and $bN$ turns blue, i.e., all four edges along the outer facet of $Q$ are of same color that is the main character (b1) of Type A given in Theorem~\ref{thm:EPediamond}. Thus the necessary conditions (b1), (b2) and (b3) shall come all together, because we assume $EP\in e\mathcal{MPGN}$.  
Now there must be two Kempe chains $K_r$ and $K_g$ as the character (b2) of Type A.       
   \end{proof}

   \begin{remark}
The existence of Type A $e$-diamond for every $e\in EP$ and the picture of Type A provide a new proof for Corollary~\ref{RGB1-thm:V5more2}(b).
   \end{remark}

   \begin{remark}
There is another way to prove this theorem by using the concept of block graphs. Given the right graph in Figure~\ref{fig:AtypeBtype}, we can build the red block graph from $EP-\{e\}$ as the first line in Figure~\ref{fig:treeofQ2}. Using what we just learned in Subsection~\ref{sec:LearnECS}, we have 
   \begin{eqnarray*}
T^a_{rgb}\quad \text{(Type A)}&  \overset{\text{\tiny ECS on $rCL(aN)$}}{\Large \Longleftrightarrow}& T^b_{rgb} \quad  \text{(Type B)}\\
& \overset{\text{\tiny ECS on $rCL(aS)$}}{\Large \Longleftrightarrow}& T^c_{rgb}\ \equiv\ T^a_{rgb} \quad  \text{(Type A)},  
   \end{eqnarray*}
where $T^c_{rgb}$ has the four surrounding edges of $e$ all green.
In the last line in Figure~\ref{fig:treeofQ2}, we skip $rBG(EP-\{e\};T^c_{rgb})$, but we show more details about exchange between $T^b_{rgb}$ and $T^c_{rgb}$.

   \begin{figure}[h]
   \begin{center}
   \includegraphics[scale=0.87]{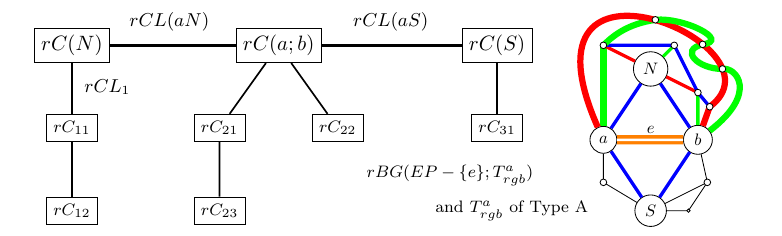}\\
   \includegraphics[scale=0.87]{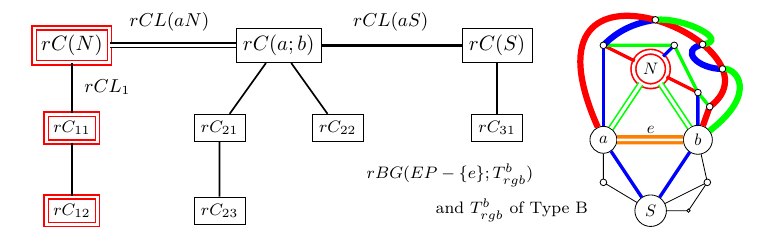}\\
   \includegraphics[scale=0.87]{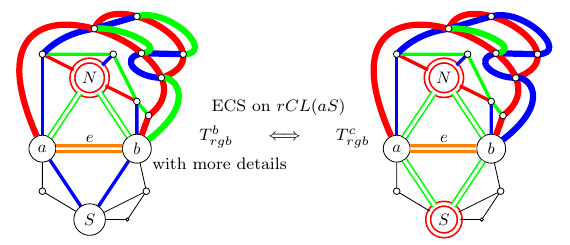}   
   \end{center}
   \caption{Top line: VCS on $rC(v_2)$; Bottom line: ECS on $rCL(v_1v_2)$} \label{fig:treeofTypeAB}
   \end{figure}

   \end{remark}

Theorem~\ref{thm:EPediamond} nearly offers sufficient conditions for $EP\in e\mathcal{MPGN}$. We will complete these if-and-only-if conditions as our final goal. 

Here is a direct consequence of Theorem~\ref{thm:EPediamond2}. 
   \begin{theorem}[Important]  \label{thm:EPNScolorable}
Let $(EP;e)$ and the $e$-diamond set generally as in Theorem~\ref{thm:EPediamond}.  
   \begin{enumerate}
\item[(a)] The vertices $N$ and $S$ are not adjacent in $EP$.   
\item[(b)] Not only $EP-\{e\}$ but also $EP-\{e\}\cup\{NS\}$ is 4-colorable.      
   \end{enumerate}
   \end{theorem}
   \begin{proof}
(a): If $N$ and $S$ are adjacent in $EP$, then either $EP=K_4$ or the triangle $N$-$a$-$S$-$N$ forms a non-trivial 3-cycle in $EP$. Both are impossible. Please, refer to Lemma~\ref{RGB1-thm:nontrivial3}. Thus, $N$ and $S$ are not adjacent in $EP$. 

(b): Just look at Type B. Let us assign the edge $NS$ red. Because the pre-existing red Kempe chain prevents a new red cycle passing through $NS$, this new R-tiling on $EP-\{e\}\cup\{NS\}$ has no red odd-cycle. Therefore, $EP-\{e\}\cup\{NS\}$ is 4-colorable.  
   \end{proof}  
   
It can be easily prove by induction that any MPG, say $G$, has $3|G|-6$ edges. Let $\omega:=|EP|$. The next corollary is just for fun.
   
   \begin{corollary}
Through the modification $EP-\{e\}\cup\{NS\}$, there are $3\omega-6$ MPG's which are 4-colorable and only different from $EP$ with only one single edge.      
   \end{corollary}

\section{Necessary and sufficient conditions for $EP$}  \label{sec:NecSufConds}

In addition to our discussion in the last two sections, the Type A and Type B still have some other new characters for $EP$ to explore. We plan to write a much precise statement for these characters. The most important thing is to accomplish the title of this section. 

Given a $2n$-semi-MPG, $M$, an R-tiling (or G-/B-tiling) is \emph{perfect} if no edge along the outer facet $2n$-gon is red. We use the word \emph{perfect} because no red half-tile is used, i.e., the tiling is made all by red diamonds. Briefly we use ``R-tiling$^\ast$'' as the abbreviation of ``R-tiling without any red odd-cycle.''  Particularly we would like discuss a 4-semi-MPG, $Q$, with it outer facet $\Omega:=N$-$a$-$S$-$b$-$N$. Notice that most of time we set $Q:=EP-\{ab\}$ for any fixed $e:=ab\in E(EP)$, but now we assume $Q$ a general 4-semi-MPG with $|Q|\le \omega$.

Let us recall the notation $T_r(Q)$ and $T_{rgb}(Q)$ of an R-tiling and an RGB-tiling on $Q$, where $T_{rgb}(Q)$ means the coexistence of R-, G- and B-tilings. If we obtain an R-tiling$^\ast$ $T_r(Q)$ first, then we can extend it to a $T_{rgb}(Q)$. We can also use $T_{rg}(Q)$, $T_{rb}(Q)$ and $T_{gb}(Q)$; however they are no different from $T_{rgb}(Q)$, because once two tilings coexist a tiling of the third color is immediately ready. Because $Q$ is a $4$-semi-MPG and an R-tiling$^\ast$ $T_r(Q)$ on One Piece is always grand, a coexisting $T_{rgb}(Q)$ extended from $T_r(Q)$ must induce a 4-coloring function on $Q$. For the detail, refer to Theorem~\ref{RGB1-thm:4RGBtiling} and Theorem~\ref{RGB1-thm:4RGBtilingGeneralization}.

Let variables $\mathcal{X}$ and $y$ denote brief names of one edge-color from red, blue and blue, and most of time $\mathcal{X}$ and $y$ are a same color. We define the following collections of tilings on $Q$ (not on $EP$). These collections have a general notation $\mathcal{XT}_{ky}(Q)$ or simply $\mathcal{XT}_{ky}$ with $Q$ assigned already, where $k\in \{0,2,4\}$. Clearly, if $\{\mathcal{X}, y\} \overset{\text{\tiny syn}}{=}\{\mathcal{X'}, y'\}$ (either both one color or both two colors) then $\mathcal{XT}_{ky} \overset{\text{\tiny syn}}{=} \mathcal{X'T}_{ky'}$, where $\overset{\text{\tiny syn}}{=}$ is the equivalence relation of synonym. 
   \begin{eqnarray*}
\mathcal{RT}_{0r} & = & \{T_r(Q): \text{a perfect R-tiling$^\ast$, i.e., all edges of $\Omega$ is black}\};\\
\mathcal{GT}_{2g} & = & \{T_g(Q): \text{a G-tiling$^\ast$ s.t.\ $\Omega$ has two green and two black}\};\\ 
\mathcal{BT}_{4b} & = & \{T_b(Q): \text{a B-tiling$^\ast$ with all four edges along $\Omega$ blue}\}.
   \end{eqnarray*}
We can define the corresponding collections that are extended from the last three:
   \begin{eqnarray*}
\mathcal{RGBT}_{0r} & = & \{T_{rgb}(Q): \text{an RGB-tiling with no red along $\Omega$}\};\\
\mathcal{RGBT}_{2g} & = & \{T_{rgb}(Q): \text{an RGB-tiling s.t.\ $\Omega$ has two green}\};\\ 
\mathcal{RGBT}_{4b} & = & \{T_{rgb}(Q): \text{an RGB-tiling with all four edges along $\Omega$ blue}\}.
   \end{eqnarray*}
  
Also recall the definition of the north-$\wedge$ edges and the south-$\vee$ edges of $\Omega$.  Additionally we define the east-$<$ to be $\{aN,aS\}$, the west-$>$ to be $\{bN,bS\}$, the double-slash-$//$ to be $\{aN,bS\}$, the double-backslash-$\backslash\backslash$ to be $\{aS,bN\}$. According to these six different pair of edge sets, we can divide $\mathcal{GT}_{2g}$ into six sub-collections. In the following us just pick three of them to write definition precisely.    
   \begin{eqnarray*} 
\mathcal{GT}^\wedge_{2g} & = & \{T_g(Q)\in \mathcal{GT}_{2g}: \text{only the north-$\wedge$ edges are green}\};\\
\mathcal{GT}^{<}_{2g} & = & \{T_g(Q)\in \mathcal{GT}_{2g}: \text{only the east-$<$ edges are green}\};\\
\mathcal{GT}^{//}_{2g} & = & \{T_g(Q)\in \mathcal{GT}_{2g}: \text{only the double-slash-$//$ edges are green}\}
   \end{eqnarray*}

Clearly, $\mathcal{GT}_{2g}=\bigcup_{x\in D} \mathcal{GT}^{\ x}_{2g}$ where $D=\{\wedge, \vee, <, >, //, \backslash\backslash \}$. Let us use $\langle\cdot\rangle$ to denote the group of synonyms, for instance $\langle\mathcal{RT}_{0r}\rangle=\mathcal{RT}_{0r}\cup \mathcal{GT}_{0g}\cup \mathcal{BT}_{0b}$ and $\langle\mathcal{GT}_{2r}\rangle=
\mathcal{RT}_{2g}\cup 
\mathcal{RT}_{2b}\cup
\mathcal{GT}_{2r}\cup
\mathcal{GT}_{2b}\cup
\mathcal{BT}_{2r}\cup
\mathcal{BT}_{2g}$. According to the discussion in the last two sections, Type A associates with $\langle\mathcal{BT}_{4b}\rangle$, Type B associates with $\langle\mathcal{GT}^{\wedge}_{2g}\rangle=\langle\mathcal{GT}^{\vee}_{2g}\rangle$ and the union set of Type A and Type B associates with $\langle\mathcal{RT}_{0r}\rangle$, i.e., $\langle\mathcal{RT}_{0r}\rangle= \langle\mathcal{BT}_{4b}\rangle \cup \langle\mathcal{GT}^{\wedge}_{2g}\rangle$. The impossible Type C for $EP-\{ab\}$ is  $\langle\mathcal{GT}^{<}_{2g}\rangle=\langle\mathcal{GT}^{>}_{2g}\rangle$ and also impossible Type D is $\langle\mathcal{GT}^{//}_{2g}\rangle=\langle\mathcal{GT}^{\backslash\backslash}_{2g}\rangle$.
We also let $\mathcal{RGBT}^x_{2g}$ to be the extension of $\mathcal{GT}^x_{2g}$ for $x\in\{\wedge, \vee, <, >, //, \backslash\backslash \}$ 

   
   \begin{figure}[h]
   \begin{center}
   \includegraphics[scale=1.2]{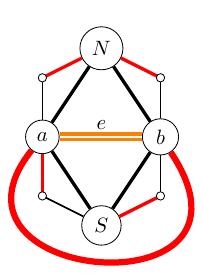}
   \includegraphics[scale=1.2]{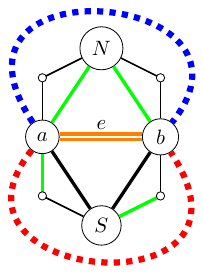}
   \includegraphics[scale=1.2]{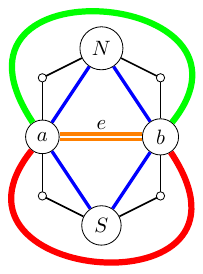}
   \end{center}
   \caption{Sample elements in $\mathcal{RT}_{0r}$, $\mathcal{GT}^\wedge_{2g}$ and $\mathcal{BT}_{4b}$ with $Q:=EP-\{e\}$} \label{fig:threeSetsTilings}
   \end{figure}

   \begin{theorem}[The Second Fundamental Theorem v1]  \label{thm:4ColorableIfandOnlyIf}
Let $M$ be an MPG and $e=ab\in E(M)$; also let $Q:=M-\{e\}$. The graph $M$ is 4-colorable if and only if  $\mathcal{GT}^{<}_{2g}\cup \mathcal{GT}^{//}_{2g}$ is non-empty. 
   \end{theorem}
   \begin{proof}
Without loss of generality, this 4-coloring function is Co[$a$:1, $N$:4, $b$:3, $S$:2 or 4]. We have $\mathcal{GT}^{//}_{2g}$ non-empty if and only if Co[$S$:2]; also we have $\mathcal{GT}^{<}_{2g}$ non-empty if and only if Co[$S$:4]. The proof is complete.     
   \end{proof}
\noindent
This theorem is very simple. However, practically this condition appears too rare to be encountered and checked.  Still we I have great respect for this theorem as a background of the coming new properties.
   \begin{corollary}  \label{thm:4ColorableIfandOnlyIf2}
Let $M$ be an MPG with $|M|\le \omega$ and $e=ab\in E(M)$; also let $Q:=M-\{e\}$. The graph $M\in e\mathcal{MPGN}$ if and only if  $\mathcal{GT}^{<}_{2g}\cup \mathcal{GT}^{//}_{2g}$ is empty.   
   \end{corollary}
   \begin{proof}
We need  $|M|\le \omega$ because we need $Q:=M-\{e\}$ 4-colorable and the domain $\mathcal{RGBT}(M)$ to check is always non-empty.
   \end{proof}
   
Let us temporary assume $Q:=EP-\{e\}$. Referring to the right graph in Figure~\ref{fig:threeSetsTilings}, we see all four edges along $\Omega$ blue. That means diamonds $aN$ and $bN$ are overlapping on $\triangle abN$, and so are diamonds $bS$ and $bS$ overlapping on $\triangle abS$.  If we extended this $T_b$ to a $T_{rgb}$, then we shall have Kempe chains $K_r|_a^b$ and $K_g|_a^b$ due to Theorem~\ref{thm:EPediamond}(b).       
Referring to the left graph in Figure~\ref{fig:threeSetsTilings}, we see all four edges along $\Omega$ black. If we extended this $T_r$ to a $T_{rgb}$, then there are two possible coloring along along $\Omega$: either Type A or Type B. By Theorem~\ref{thm:EPediamond}(b) and (c), a Kempe chain $K_r|_a^b$ is guaranteed.   
Finally let us refer to the middle graph in Figure~\ref{fig:threeSetsTilings} which is an element in $\mathcal{GT}^\wedge_{2g}$. If we extended this $T_g$ to a $T_{rgb}$, then there are two possible coloring on edges $aS$ and $bS$: either both red or both blue. Both red implies a Kempe chain $K_b|_a^b$, and both blue implies a Kempe chain $K_r|_a^b$.  


   \begin{theorem}[The Second Fundamental Theorem v2:  the surrounding four edges of $e$-diamond; some necessary conditions]  \label{thm:R0G2B4}
Let $EP\in e\mathcal{MPGN}$ and any $e=ab\in E(EP)$, where $e$-diamond has its 4-cycle $\Omega:N$-$a$-$S$-$b$-$N$ around. Let $Q=E-\{ab\}$ which is a 4-semi-MPG with its outer facet $\Omega$. All the following statements are true:
   \begin{itemize}
\item[(a)] The sets $\mathcal{RT}_{0r}(Q)$, $\mathcal{GT}^\wedge_{2g}(Q)$ and $\mathcal{BT}_{4b}(Q)$  are all non-empty. Also we have $\langle\mathcal{GT}^\wedge_{2g}(Q)\rangle= \langle\mathcal{GT}_{2g}(Q)\rangle$, i.e., $\langle\mathcal{GT}^{<}_{2g}(Q)\rangle$ and $\langle\mathcal{GT}^{//}_{2g}(Q)\rangle$ are empty.
\item[(b)] For every $T_r\in \mathcal{RT}_{0r}(Q)$, there exist a red path of even length from $a$ to $b$.
\item[(c)] For every $T_g\in \mathcal{GT}^\wedge_{2g}(Q)$, there must exist an extension $T_{rgb}$ of $T_g$ with the south-$\vee$ edges blue (or red by symmetry) and then there must exist a red (blue) path of even length from $a$ to $b$.
\item[(d)] For every  $T_b\in \mathcal{BT}_{4b}(Q)$, there must exist an extension $T_{rgb}$ of $T_b$, and then there must exist a red path and a green path from $a$ to $b$ (that are both even length). 
   \end{itemize}
   \end{theorem}
   \begin{proof}
The background of these four claims are Theorem~\ref{thm:EPediamond}.    
   
(a): Notice that $\langle\mathcal{BT}_{4b}\rangle=\mathcal{RT}_{4r}\cup \mathcal{GT}_{4g}\cup \mathcal{BT}_{4b}$. By Theorem~\ref{thm:EPediamond}(a), $\langle\mathcal{BT}_{4b}\rangle$ is non-empty, so is $\mathcal{BT}_{4b}$.  The same argument works for $\mathcal{RT}^\wedge_{2r}$, $\mathcal{GT}^\wedge_{2g}$, $\mathcal{BT}^\wedge_{2b}$, $\mathcal{RT}^\vee_{2r}$, $\mathcal{GT}^\vee_{2g}$, and $\mathcal{BT}^\vee_{2b}$, because all together of them form $\langle\mathcal{GT}^\wedge_{2g}\rangle$  which consists of all RGB-tilings of Type B w.r.t.\ $(EP; e)$. 

Clearly tilings in $\mathcal{GT}^{<}_{2g}$ and $\mathcal{GT}^{>}_{2g}$ are Type C
and tilings in $\mathcal{GT}^{//}_{2g}$ and $\mathcal{GT}^{\backslash\backslash}_{2g}$ are Type D. All these four sets as well as Types C, D are empty sets. Therefore, $\langle\mathcal{GT}^{<}_{2g}(Q)\rangle$ and $\langle\mathcal{GT}^{//}_{2g}(Q)\rangle$ are empty.    

(b), (c), (d): By Lemma~\ref{RGB1-thm:RtilingOnePiece}, the tilings $T_r$ in (b), $T_g$ in (c) and $T_b$ in (c) must be grand. Additionally, no odd-cycle make these three single color tilings inducing a 4-colorable function by Lemma~\ref{RGB1-thm:grandRtilingNoOddCycle}. Thus, $T_r$ in (b), $T_g$ in (c) and $T_b$ can extend to their own RGB-tilings.

At this moment, to ensure $EP$ is extremum we need a proper Kempe chain $K_\ast$ such that $K_\ast \cup \{e\}$ is a non-trivial odd-cycle. Part (b) only need a red $K_r$, so it is not necessary to extend a coexisting RGB-tiling. Part (c) really need a coexisting RGB-tiling $T_{rgb}$ and by symmetry we assume the south-$\vee$ edges blue; so that we must have $K_r|_a^b$ of even length and then $K_r|_a^b\cup \{e\}$ is a red odd-cycle. Part (d) also need a coexisting RGB-tiling, and then the dual Kempe chains $(K_r, K_g)$ exist.
   \end{proof}

   \begin{remark}
(1) What is difference between Theorem~\ref{thm:R0G2B4} and Theorem~\ref{thm:EPediamond}? Both of them provide necessary conditions for $EP\in e\mathcal{MPGN}$, but Theorem~\ref{thm:R0G2B4} starts with an R, G or B-tiling rather than an RGB-tiling on $E-\{e\}$. Also Theorem~\ref{thm:R0G2B4}(b), which combines Types A and B, provides a new situation. (2) In order to claim sufficient conditions for $EP\in e\mathcal{MPGN}$, we show the following lemma first. This lemma also drops a hint: $\mathcal{BT}_{4b}(Q)$ might be empty for a general 4-semi-MPG $Q$ with $|Q|\le \omega$. We shall to careful to check something on an empty set, because nothing to check means everything is true. Therefore, we leave this particular $\mathcal{BT}_{4b}$ to the next section to discuss. (3) Our next goal is: Given an MPG $M$ with $|M|\le \omega$, how to recognize $M$ being non-4-colorable by offering only  $\mathcal{XT}_{ky}(M-\{e\})$?
   \end{remark}

For item (d), why we put "both even length" in parentheses? The next lemma is an answer.  
   \begin{lemma}  \label{thm:twoChainsEven}
In One Piece with $T_{rgb}$ provided, the existence of chains $(K_r|_\alpha^\beta,K_g|_\alpha^\beta)$ implies both chains are even length.  
   \end{lemma} 
   \begin{proof}
We have One Piece, so R-tiling is grand. In a grand R-tiling, a red path $K_r|_\alpha^\beta$ guarantees that either $\alpha, \beta \in V_{13}$ or $\alpha, \beta \in V_{24}$. Then, refer to Lemma~\ref{RGB1-thm:RandomWalk} to show $K_g|_\alpha^\beta$ is even length. Similarly we can show $K_r|_\alpha^\beta$ of even length. The given R-/G-tiling is grand is the key point of this lemma and the $n$-gon outer facet has minor impact.   
   \end{proof}

   \begin{lemma}   \label{thm:R0exist}
Let $Q$ be a general 4-semi-MPG with $|Q|\le \omega$ and its outer facet $\Omega:= N\text{-}a\text{-}S\text{-}b\text{-}N$. The sets $\mathcal{RT}_{0r}(Q)$ and $\mathcal{GT}_{2g}(Q)$ are both non-empty.
   \end{lemma}
   \begin{proof} 
Let us consider three cases: I. $Q\cup\{ab\} \in e\mathcal{MPGN}4$; II. $Q\cup\{NS\} \in e\mathcal{MPGN}4$; III. neither I nor II. 

[I and II]:  See Theorem~\ref{thm:R0G2B4}(a).
 
[III]: Along the outer facet $\Omega$ (a 4-cycle), at least a pair of vertices on the opposite are non-incident in $Q$. Without loss of generality, we say $N$ and $S$ non-adjacent, then $Q':= Q\cup\{NS\}$ is an MPG and $Q\cup\{NS\} \notin e\mathcal{MPGN}4$ by the hypothesis. Since $|Q'|=|Q|\le \omega$, we know $Q'$ is 4-colorable. Without loss of generality we have a 4-coloring function $f:V(Q')\rightarrow \{1,2,3,4\}$ with $f(N)=1$ and $f(S)=3$, i.e., the edge $NS$ is red and $\Omega\cup\{NS\}$ is a red diamond. The corresponding edge-coloring induced by $f$ is an RGB-tiling on $Q'$ and then an R-tiling$^\ast$ on $Q'$ with $NS$ red. So, $\mathcal{RT}_{0r}$ is non-empty.  An RGB-tiling on $Q'$ with $NS$ red also induce a $T_g\in \mathcal{GT}_{2g}$.  So, $\mathcal{GT}_{2g}$  is non-empty.
   \end{proof}

   \begin{remark}
We get benefit from $Q$ as One Piece: any R-/G-/B-tiling$^\ast$ on $Q$ can extend to a RGB-tiling on $Q$ and then a 4-coloring function. Once $\mathcal{GT}_{2g}(Q)$ is non-empty, so is $\mathcal{RGBT}_{2g2b}(Q)$. Notice that $\mathcal{RGBT}_{0r}(Q)= \mathcal{RGBT}_{4g}(Q)\cup \mathcal{RGBT}_{4b}(Q)\cup \mathcal{RGBT}_{2g2b}(Q)$; so $\mathcal{RGBT}_{0r}(Q)$ is non-empty. Finally, $\mathcal{RT}_{0r}(Q)$ is non-empty. Now we can say that $\mathcal{GT}_{2g}(Q)$ plays the key role of Lemma~\ref{thm:R0exist}.      
   \end{remark}   

   \begin{example} \label{ex:BT4bempty}
Consider $K_4$ with vertex set $\{a,b,N,S\}$ and let $Q:= K_4 -\{ab\}$. Clearly $\mathcal{RT}_{0r}$ and $\mathcal{GT}_{2g}$ are both non-empty. But $\mathcal{BT}_{4b}$ is empty. 
   \end{example}

\subsection{Types B, C and D: two greens and two blues}

Two greens and two blues along the four surrounding edges of $e$-diamond is the commend character of Types B, C and D. Thus the next theorem comes up naturally.      
   
   \begin{theorem}[The Second Fundamental Theorem v3: necessary and sufficient conditions]  \label{thm:R0suf}
Given any MPG, denoted by $M$, with $|M|\le \omega$, here is a necessary and sufficient condition for $M\in  e\mathcal{MPGN}4$:\quad {\bf \{There exists an edge (a sufficient condition)/ It is true for every edge (a necessary condition)\}} $e=ab$ in $M$ such that the 4-semi-MPG $Q:=M-\{e\}$ satisfies one of the following items:     
   \begin{itemize}
\item[(i)] For every $T_r\in\mathcal{RT}_{0r}(Q)$, there is a red $a$-$b$ path $K_r$ of even length in $T_r$.
\item[(ii)] For every $T_g\in\mathcal{GT}_{2g}(Q)$, let us first extend it to an RGB-tiling $T_{rgb}(Q)$. By symmetry or synonym relation, we assume that the four surrounding edges of $e$ are two green and two blue; then there exists a non-trivial red $a$-$b$ path $K_r$ of even length in $T_{rgb}$.
\item[(ii')] For every $T_g\in\mathcal{GT}_{2g}(Q)$, the $e$-diamond must be Type B after we extend this $T_g$ to be any possible $T_{rgb}$. In other word, we have $\langle \mathcal{GT}_{2g}(Q)\rangle= \langle\mathcal{GT}^\wedge_{2g2b}(Q)\rangle$ (i.e., $\langle\mathcal{GT}^{<}_{2g}(Q)\rangle =\langle\mathcal{GT}^{//}_{2g}(Q)\rangle =\emptyset$) and we obtain $K_r|_a^b$.    
   \end{itemize}
   \end{theorem}
   
   \begin{remark}
First of all, the statement of Theorem~\ref{thm:R0suf}(ii') does not mention that $K_r|_a^b$ is even length. This fact is automatically true for $a$-$N$-$b$ being a same color and by Lemma~\ref{RGB1-thm:evenoddRGB}(b). The situation and the reason this time are different from Lemma~\ref{thm:twoChainsEven}. Second,  Theorem~\ref{thm:R0G2B4}(a) and (b) (as well as (a) and (c); (a) and (d) respectively) offer a necessary condition for $EP$. Clearly (a) and (b) associate with item (i) here;  (a) and (c) associate with item (ii) and (ii'); but (a) and (d) have no corresponding item here, because this corresponding and powerful item is a big project which need a whole new section to study and explain.    
   \end{remark}

   \begin{remark}
We write {\bf \{There exists an edge (sufficient)/ It is true for every edge (necessary)\}} in this theorem. It is weird to see that a sufficient condition is weaker than a necessary condition. How this theorem comes to be if-and-only-if conditions? The reason or this phenomenon is kind of ``one (diamond) for all and all for one'' as follows:
   \begin{eqnarray*}
&& \text{A particular $e$-diamond satisfies any item of Theorem~\ref{thm:R0suf}.}\\    
&\Rightarrow & \text{$EP\in e\mathcal{MPGN}4$.}\\
&\Rightarrow & \text{Any edge in $EP$ plays the same role as this $e$-diamond by Theorem~\ref{RGB1-thm:V5more}(a).}\\
&\Rightarrow & \text{It is true for every edge in $EP$ such that... (all necessary conditions).}
   \end{eqnarray*}  
   \end{remark}   
   \begin{remark}    \label{re:LastAttempt}
The hypothesis $|M|\le \omega$ is important, because we need $Q:=M-\{e\}$ 4-colorable and the domain $\mathcal{RGBT}(M)$ to check is always non-empty.    
   \end{remark}

   \begin{proof}
First of all, this MPG $M$ with $|M|\le \omega$ implies $Q:=M-\{e\}$ 4-colorable. By Lemma~\ref{thm:R0exist},
the sets $\mathcal{RT}_{0r}(Q)$ and $\mathcal{GT}_{2g}(Q)$ are both non-empty. Also by extending single color tilings to RGB-tiling, we have $\langle\mathcal{RT}_{0r}(Q)\rangle=\langle\mathcal{GT}_{2g}(Q)\rangle \cup \langle\mathcal{BT}_{4b}(Q)\rangle$, even though we do not know whether $\mathcal{BT}_{4b}(Q)$ is empty or not.  
   
(i): By Theorem~\ref{thm:R0G2B4}(a) and (b), this item is a necessary condition. Now let us prove this item is a sufficient condition. By Lemma~\ref{thm:R0exist}, $\mathcal{RT}_{0r}(Q)$ is non-empty, thus discussing a $T_r\in\mathcal{RT}_{0r}(Q)$ is logically reasonable to go ahead. To possibly 4-color $M$ we shall try every R-tiling in  $\mathcal{RT}_{0r}(M-\{e\})$ first in order to try the last red $e$-diamond. Yes, the last $e$-diamond is the final judge: Whether the only possible red odd-cycle turns out? It does turn out if and only if we see a $K_r|_a^b$ of even length for every $T_r\in\mathcal{RT}_{0r}(Q)$.   

(ii'): By Theorem~\ref{thm:R0G2B4}(a) and (c), this item is a necessary condition. Now let us prove this item is a sufficient condition. Actually Corollary~\ref{thm:4ColorableIfandOnlyIf2} is a good reference for if-and-only-if. Item (ii') provides $\langle \mathcal{GT}_{2g}(Q)\rangle= \langle\mathcal{GT}^\wedge_{2g2b}(Q)\rangle$ (i.e., $\langle\mathcal{GT}^{<}_{2g}(Q)\rangle =\langle\mathcal{GT}^{//}_{2g}(Q)\rangle =\emptyset$) that is enough for sufficient condition. But without $K_r|_a^b$ means there exists $K_r|_S^N$ which will transform $T_g^\wedge$  to be $T_g^<$.    

(ii') $\Leftrightarrow$ (ii): The direction (ii') $\Rightarrow$ (ii) is trivial. Now let us show (ii') $\Leftarrow$ (ii). We need only show the position of 
two green and two blue along $Omega$. The reason dues to the even length $K_r|_a^b$. For even length, these two green edges must be either the north-$\wedge$ edges or the south-$\vee$ edges. Therefore, the $e$-diamond must be Type B.   
   \end{proof}

   \begin{remark}
If we assume $e\mathcal{MPGN}4$ non-empty, then $\omega$ is well-defined and $|M|< \omega$ implies $M$ 4-colorable. So all these sufficient conditions are good to check (to distinguish) the two kinds of $M$ with $|M|= \omega$. However, we still consider it is good description to including every $M$ with $|M|< \omega$ in this theorem; because once it satisfies any item of Theorem~\ref{thm:R0suf} we can conclude that either this $M$ should not exist or the set $\mathcal{XT}_{ky}(Q)$ of single color tilings should be empty. Any contradiction is always what we hope for. How about we remove  the requirement $|M|\le \omega$ and set no limit on $|M|$. The problem is that we have no idea about a general MPG, which is not an $EP$, in $ \mathcal{N}4$. We believe that those necessary conditions in Theorem~\ref{thm:R0G2B4} and Theorem~\ref{thm:R0suf} would not work for this general MPG in $ \mathcal{N}4$. So $|M|\le \omega$ is important and cannot be relaxed.
   \end{remark}

   \begin{remark}
Here is a interesting question: For item (ii'), what happens if we 
have $\langle\mathcal{GT}^{<}_{2g2b}(Q)\rangle =\langle\mathcal{GT}^{//}_{2g2b}(Q)\rangle =\emptyset$ but know nothing about the existence of $K_r|_a^b$? The argument is easy. We claim that if $\langle\mathcal{GT}^{<}_{2g2b}(Q)\rangle =\langle\mathcal{GT}^{//}_{2g2b}(Q)\rangle =\emptyset$, then $K_r|_S^N$ is impossible to appear in any RGB-tiling on $EP-\{e\}$. Because given ${T_{rgb}^{\wedge}}(Q)_{2g2b}$ and $K_r|_S^N$, we will then have ${T_{rgb}^{//}}(Q)_{2g2b}$; given ${T_{rgb}}(Q)_{4b}$ and $K_r|_S^N$, we will then have ${T_{rgb}^<}(Q)_{2g2b}$. Therefore, $K_r|_a^b$ must exists for any ${T_{rgb}}(Q)$. However, the hypothesis that {\bf given any ${T_{rgb}}(Q)_{4b}$ and then we always see $(K_r|_a^b, K_g|_a^b)$} is just a necessary condition but not enough to be a sufficient condition for $Q$ non-4-colorable. We will show two counterexamples in the next section.
   \end{remark}

\subsection{If-and-only-if condition by odd-cycle and some conjectures}
It is nice to achieve several if-and-only-if conditions for $EP\in e\mathcal{MPGN}4$ by Theorems~\ref{thm:4ColorableIfandOnlyIf2} and~\ref{thm:R0suf}. However, the first one is easy to check but rare to encounter and the second one is hard to check ``every... must be'' but we see it a lot of times.  The following corollary claim three conditions involving odd-cycles.

   \begin{corollary}
\label{thm:R0_OneOddCycle}
Let $M$ be an MPG with at least an R-tiling. (We exclude out that case that $M$ has no R-tiling. In this case $M$ is definitely non-4-colotable.)
   \begin{itemize}
\item[(a)] A sufficient and necessary condition for $M$ non-4-colorable is that any R-tiling on $M$ has at least one red odd-cycle.    
\item[(b)] Based on (a), $M\in  e\mathcal{MPGN}4$ if and only if there an R-tiling on $M$ who has exactly one red odd-cycle (or $oc(T_{r})=1$).  
\item[(c)] Based on (b), if we fix any edge $e=ab \in E(M)$, then there exists an R-tiling on $M$ whose single red odd-cycle passing through $e$. 
   \end{itemize}
   \end{corollary}
   \begin{proof}
Briefly, (a) $\Leftrightarrow$ Theorem~\ref{RGB1-thm:4RGBtiling}, and (c) $\Leftrightarrow$ Theorem~\ref{RGB1-thm:eMPG4} together with Theorem~\ref{RGB1-thm:4RGBtiling}.
   
As for part (b), we use an unclear concept. We would rather say that the following discussion is not a proof but a definition of ``exactly one red odd-cycle'' or the definition of $oc(T_{r})$.
   
An independent cycle $C$ means a collection $E(C)$ of edges such that $C$ is 2-connected and $C-\{e\}$ is 1-connected for any $e\in E(C)$. For example, we see two independent cycles in  the first graph of Figure~\ref{fig:indepCyc}. They are one independent odd-cycle and one independent even cycle.

   \begin{figure}[h]
   \begin{center}
   \includegraphics[scale=0.9]{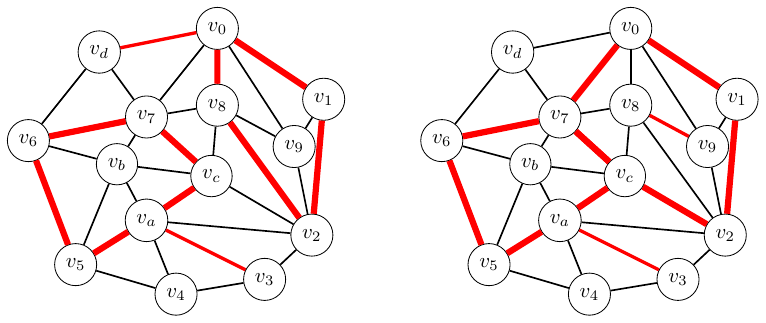}
   \end{center}
   \caption{Left: two independent cycles; Right: a combination of two cycles which has $oc(T_{r})=1$.}  \label{fig:indepCyc}
   \end{figure}
A combination of cycles looks like the second graph of Figure~\ref{fig:indepCyc}. We use this graph to explain. Obviously there are three cycles, namely   
	\begin{eqnarray*}
C_1 &:=& v_0\text{-}v_1\text{-}v_2\text{-}v_c\text{-}v_7\text{-}v_0;\\
C_2 &:=& v_5\text{-}v_6\text{-}v_7\text{-}v_c\text{-}v_a\text{-}v_5;\\
C_3 &:=& v_0\text{-}v_1\text{-}v_2\text{-}v_c\text{-}v_a\text{-}v_5\text{-}v_6\text{-}v_7\text{-}v_0.
	\end{eqnarray*}
Notice that $C_1$ and $C_2$ are odd-cycles, and $C_3$ is an even-cycle. However, we wound not say this R-tiling has two odd-cycles. Actually this R-tiling has only one odd-cycle. We will explain the reason behind.

Let $EC$ be the edge set of this combination and $EC$ consists of 9 edges. Also define the pair $(\# o,\# e )_{EC}$ to be the numbers of odd-cycles and even-cycles made by $EC$. Our example has $(\# o,\# e )_{EC}=(2,1)$.

Among the nine edges in out $EC$, $v_7v_c$ is different from the other eight; because $(\# o,\# e )_{EC-\{v_7v_c\}}=(0,1)$ and $(\# o,\# e )_{EC-\{v_xv_y\}}=(1,0)$ when $v_xv_y$ is one of the other eight. The pairs $(0,1)$ and $(1,0)$ indicate an independent cycle remained. Therefore, this combination is actually combined by two cycle. Also due  to $(0,1)$ and $(1,0)$, it shall be combined by one odd-cycle and one even-cycle. Thus, we would say this combination is $C_1\oplus C_3$ or $C_2\oplus C_3$; but we would not say it is $C_1\oplus C_2$. Now we can conclude that the combination in the second graph of Figure~\ref{fig:indepCyc} has exactly one odd-cycle. 

Please, refer to the formal definition behind. Because choosing any $e\in EP$ we can have an RGB-tiling $T_{rgb}$ on $EP-\{e\}$ and $T_{r}\cup \{e\}$ has exactly one red odd-cycle, where $T_{r}$ is the R-tiling induced by $T_{rgb}$.

   \end{proof}   

	\begin{definition}
Let $T_{r}$ be an R-tiling on an MPG $M$. This $T_{r}$ is also the set of all red edges. The number of odd cycles of $T_{r}$ is defined to 
  $$
oc(T_{r}) := \min \{|S| \mid S\subseteq T_{r}
\text{ such that $(\# o,\# e )_{T_{r}-S}=(0,\ast)$}\}.
  $$
Similarly we can define the number of even cycles $ec(T_{r})$. 
	\end{definition}

Let us choose an $e$-diamond of $EP$ and a fixed RGB-tiling $T_{rgb}(EP-e)$, whether $T_{rgb}$ is Type A or Type B.  Also let $FP$ and this $e$-diamond look like either graphs in Figure~\ref{fig:AtypeBtype00}. There are a north and a south red canal rings, denoted by $NrCL_e$ and $SrCL_e$ respectively. Indeed, $NrCL_e:= rCL(aN)$ and $SrCL_e:= rCL(aS)$. Without loss of generality, we assume $Co(a,b)=1$ and then we have a fixed direction of current along $NrCL_e$ and $SrCL_e$. We also use $NrCL_e^r$ and $SrCL_e^r$to denote the right canal banks of these two currents. Please, refer to Definition~\ref{RGB1-def:deja_vu}.   

The union $NrCL_e^r \cup SrCL_e^r$ has three parts. (1): Let $rNS(EP;e;T):=NrCL_e^r \cap SrCL_e^r$.
Clearly $e\in rB(EP;e;T)$; (2): Let $rDJV(EP;e;T)$ consist of those \emph{deja-vu} edges for $NrCL_e^r$ or for $SrCL_e^r$; (3) Besides $rNS(EP;e;T)$ and $rDJV(EP;e;T)$, the rest of $NrCL_e^r \cup SrCL_e^r$ forms some {\bf red even-cycles}. Furthermore, these cycles are classified into three sub-parts. (3a): Cycles made by both edges from $NrCL_e^r$ and $SrCL_e^r$; (none of them from $rB(EP;e;T)$.)  (3b): Cycles made by edges from $NrCL_e^r$; (3c): Cycles made by both edges from $SrCL_e^r$. 

	\begin{remark} \label{re:eToAnothere}
For (1), every edge $e'$ in $rNR(EP;e;T)$ can play the same role as $e$ for this fixed $T_{rgb}$. In other words, we can turn $e$ from yellow to red and in the same time turn $e'$ from red to yellow. Of course, we need to perform ECS along $NrCL_e$ and $SrCL_e$ (not all but parts of them). Finally we have a new $T'_{rgb}$ with a Type A or Type B $e'$-diamond.  Usually Type A is good, because  to determine at least a new $K'_g$ such that $K'_g\cup\{e'\}$ is an odd-cycle.
	\end{remark}

	\begin{remark}
For (3) and any other red even-cycle in  this $T_{rgb}(EP-e)$ form their own normal red canal rings. Performing ECS on any of these red canal rings or even combination of them will of course do nothing on $K_r$ in $\Sigma'$; However, amazingly these ECS might change the real shapes of $K_g$ and $K_r$, but nothing to do with the original green/blue connection in view of skeleton in $\Sigma'$.
	\end{remark}

	\begin{remark}
For (2), \emph{deja-vu} edges in $NrCL_e^r$ (or $SrCL_e^r$) is a magic, because it can be a short cut of the  current along $NrCL_e^r$. If we cross a red deja-vu edge when we perform ECS designed in Remark~\ref{re:eToAnothere}, then we will get two new $e$-diamonds. Two $e$-diamonds at the same might offer some interesting results.
	\end{remark}

Fix a Type A $e$-diamond of $EP$ and we concern about all kind of the RGB-tilings $T_i(EP-e)$ who has exactly one red (green) odd-cycle if we replace the yellow double-line $e$ by red (green).

Let us modify $EP$ by merging $a=b$ as well as merging $aN=bN$ and $aS=bS$, and we obtain a new MPG, denoted by $EP^{a=b}$ (no more vertex $b$). Clearly $\deg(a)\ge 6$ and $T_{rgb}$

 such that once $aN$ and $aS$

 or Type B

These if-and-only-if conditions for $EP$ will definitely make contribution in the further studies.  Theorem~\ref{thm:R0G2B4} and Theorem~\ref{thm:R0suf} offer a new approach for proving the Four Color Theorem without checking by a computer. According to the discussion from Section~\ref{sec:RGBKempeChainDeg5} up to here, we summarize the main idea as follows.
   \begin{quote}
Compared with the classical Kempe proof that considered 4-colorable $EP-\{v\}$ with $\deg(v)=5$, our new approach study 4-colorable $EP-\{e\}$ for any edge $e$ in $EP$.
   \end{quote}

\section{Type A is just a syndrome}

From the successful if-and-only conditions provided in the last section, the next property is highly recommended. 

   \begin{fconjecture}  \label{thm:4bsufnes}
Given any MPG, denoted by $M$ with $|M|\le \omega$, and provided $\mathcal{BT}_{4b}(Q)$ {\bf non-empty}, here is a necessary and sufficient condition for $M\in  e\mathcal{MPGN}4$: There exists an edge $e=ab$ in $M$ such that the 4-semi-MPG $Q:=M-\{e\}$ such that for every $T_b\in\mathcal{BT}_{4b}$ with any extended $T_{rgb}$ there exist the dual Kempe chains $(K_r|_a^b,K_g|_a^b)$ and both of them are even length.
   \end{fconjecture}

Unfortunately, we are going to demonstrate two simple but critical counterexamples as follows.

\begin{figure}[h]
\begin{center}
   \includegraphics[scale=1.2]{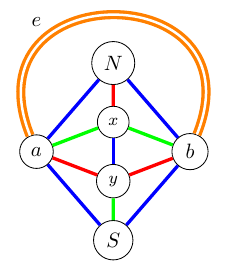}\qquad \quad
   \includegraphics[scale=1.2]{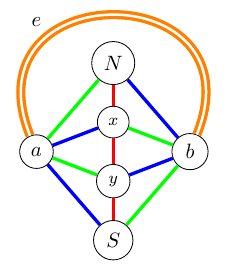}
\end{center}
\caption{$Q_1:=M_1-\{e\}$} \label{fig:TypeCounterEx}
\end{figure}   
   \begin{example}
The left graphs in Figures~\ref{fig:TypeCounterEx} and \ref{fig:TypeCounterEx2} are $M_1, M_2 \in e\mathcal{MPGN}4$.
Both graphs show an $e$-diamond of A-type. Notice that we draw  $\Sigma'$ inside and $\Sigma$ outside $\Omega:= N\text{-}a\text{-}S\text{-}b\text{-}N$ this particular time. Given the four edges along $\Omega$ are all blue, the edges that link any two inner vertices of $\Sigma'$ must be all blue, in order to fulfill the unique blue canal line. The two left graphs show the only possible RGB-tiling on $Q_1$/$Q_2$ under synonyms w.r.t.\ red and green. So $|\mathcal{BT}_{4b}(Q_1)|=|\mathcal{BT}_{4b}(Q_2)|=1$ are non-empty and we do find out $(K_r|_a^b,K_g|_a^b)$ in these two RGB-tilings.  Unfortunately the right graphs in Figures~\ref{fig:TypeCounterEx} and \ref{fig:TypeCounterEx2} show that $M_1, M_2$ are 4-colorable if we color edge $e$ by red.   
\begin{figure}[h]
\begin{center}
   \includegraphics[scale=1.2]{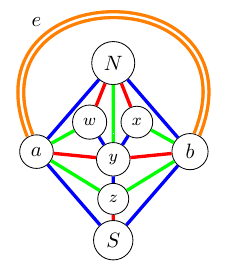}\qquad \quad
   \includegraphics[scale=1.2]{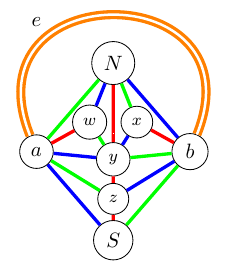}   
\end{center}
\caption{$Q_2:=M_2-\{e\}$} \label{fig:TypeCounterEx2}
\end{figure}   
   \end{example}

As the title of this section ``Type A is just a syndrome'', it dose not mean useless of this syndrome. Some thing might happen behind a syndrome. How to remedy or diagnose this particular syndrome? We suggest checking by Theorem~\ref{thm:4ColorableIfandOnlyIf} directly. Actually Type A syndrome still offers good reason to simplify the check method of Theorem~\ref{thm:4ColorableIfandOnlyIf}.      

   \begin{remark}[To be or not to be; the limit of Type A]
Conjecture~\ref{thm:4bsufnes} is incorrect but it still reveals some information for our further study. At the end of Part III of this paper (Section~\ref{RGB3-sec:deg5Adjacent}), we will provided a false proof that concerns two adjacent vertices of degree 5. That proof is incorrect because we did use Conjecture~\ref{thm:4bsufnes}. Due to that false proof, we learn more about the limit of Type A. The author has lunched a new research to find the possible approach to conquer this dilemma of to-be-or-not-to-be.  
   \end{remark}

\section{4-cycles in $EP$}

This is an independent section, but the result is very important. As an extremum MPG, $EP\ in e\mathcal{MPGN}4$ has many properties that other MPG's do not, i.e., some properties that neither 4-color graphs nor non-extremum non-4-color ones have. 

For instance, adding an extra vertex $v$ into the middle of any triangle, say $w$-$x$-$y$-$w$, of $EP$ and linking new edges from three neighbors to $v$ to get a new MPG. This new MPG, denoted $EP+v$, is of course non-4-colorable and its order is $\omega+1$. Interestingly  $\deg(v)=3$.  Please, refer to Theorem~\ref{RGB1-thm:V5more} and Theorem~\ref{RGB1-thm:V5} as for $EP$. Also $w$-$x$-$y$-$w$, a 3-facet in $EP$, is now a non-trivial triangle of $EP+v$. Please, refer to Lemma~\ref{RGB1-thm:nontrivial3}. Does $G-\{vw\}$ follow those necessary conditions given in Theorem~\ref{thm:R0G2B4}? The answer is no, because $Q:=G-\{vw\}$ is non-4-colorable and then the sets $\mathcal{RT}_{0r}$, $\mathcal{GT}_{2g}$, $\mathcal{GT}^\wedge_{2g}$ and $\mathcal{BT}_{4b}$ are all empty. Thus, Theorem~\ref{thm:R0G2B4}(a) fails for $G-\{vw\}$, not to say Theorem~\ref{thm:R0G2B4}(b), (c) and (d).  

Rather than adding just an extra vertex $v$ to $EP$, we can glue up any new MPG $M$ with $|M|\ge 5$ onto $EP$. Just let them share a common triangle $w$-$x$-$y$-$w$.  Notice that the case in the last paragraph have $V(M)=\{v,w,x,y\}$. This new non-4-colorable MPG, denoted by $EP\oplus M$, has a non-trivial triangle $w$-$x$-$y$-$w$. However, Lemma~\ref{RGB1-thm:nontrivial3} said that every triangle in $EP$ must be trivial. So, that lemma is only for $EP$ as an extremum, but not for $EP\oplus M$. In $EP\oplus M$, it is also easy to find  three non-trivial 4-cycles, where we define a \emph{trivial 4-cycle} associating with four surrounding edges of a diamond.

In the following, let us focus on 4-cycles in $EP$. A trivial 3-cycle means a 3-facet, and a trivial 4-cycle forms a diamond; both ideas of ``trivial'' indicate the kinds of cycles that have no vertex inside.

   \begin{theorem}[Important]   \label{thm:trivial4cycle}
If $\Omega:=a$-$b$-$c$-$d$-$a$ is a 4-cycle in $EP$, then either $ac$ or $bd$ is an edge of $EP$, and $a,b,c,d$ induce a single diamond in $EP$.
   \end{theorem}  
   \begin{remark}
This is again a special property for $EP$,  but not for a non-4-colorable MPG with order greater than $\omega=|EP|$. The 4-cycle $\Omega$ separate $EP$ into two regions: $\Sigma$ (inside) and $\Sigma'$ (outside) with  $\Sigma\cap\Sigma'=\Omega$. Both $\Sigma$ and $\Sigma'$ are 4-semi-MPG's, and $|\Sigma|$, $|\Sigma'|$ are less than $|EP|$ if $\Omega$ is non-trivial. Otherwise, a trivial 4-cycle $\Omega$ will make the inside of either $\Sigma$ or $\Sigma'$ no vertex; as a part of MPG $EP$, this empty inside shall have either the edge $ac$ or $bd$. On the other hand, if $ac$ exists inside $\Sigma$ as well as crosses $\Omega$, then $\Sigma$ must be a diamond by Lemma~\ref{RGB1-thm:nontrivial3}. If both $ac$ and $bd$ exist in $EP$, then by Lemma~\ref{RGB1-thm:nontrivial3} we must have $EP=K_4$, which is a contradiction. This is also the reason for ``either $ac$ or $bd$'' is an edge of $EP$. 
   \end{remark}
   \begin{proof}
Let us adopt the notation in the remark above. We will prove by contrapositive by assuming both sides of $\Omega$ have some vertices. By this assumption, neither the edge $ac$ nor $bd$ exists in $\Sigma$ and in $\Sigma'$, because the existence of $ac$ inside $\Sigma$ set the existence of the triangles $a$-$b$-$c$-$a$ and $a$-$d$-$c$-$a$, and then force $\Sigma$ to be a diamond by Lemma~\ref{RGB1-thm:nontrivial3}. Because $\Sigma\cup\{ac\}$ (draw the edge $ac$ outside $\Omega$) and $\Sigma'\cup\{ac\}$ (draw the edge $ac$ inside $\Omega$) are MGPs with $|\Sigma|$ and $|\Sigma'|$ less then $|EP|$, they are 4-colorable and then $\Sigma$  and $\Sigma'$  are also 4-colorable. 

If every 4-coloring function on $\Sigma$ makes $a,b,c,d$ exactly four different colors then the new graph $G:=\Sigma\cup\{au,bu,cu,du\}$ forms a non-4-colorable MPG.  The result either contradicts to $EP$  being an extremum, or $G=EP$ which contradicts to Theorem~\ref{RGB1-thm:V5} for $\deg(u)=4$. Therefore, there is at least a 4-coloring of $\Sigma$ that makes $a,b,c,d$ at most three different colors. This argument also works for $\Sigma'$, so there is at least a 4-coloring of $\Sigma'$  that makes $a,b,c,d$ at most three different colors.  

Suppose both $\Sigma$ and $\Sigma'$ have a 4-coloring making $a,b,c,d$ only two colors. Then $EP=\Sigma\cup\Sigma'$ is 4-colorable. So, it is a contradiction. So, at least one of $\Sigma$ and $\Sigma'$ has all its 4-colorings making $a,b,c,d$ only two colors.

Without loss of generality, we assume that every 4-coloring of $\Sigma$ never making $a,b,c,d$ only two colors.  However, we also know that there is at least a 4-coloring of $\Sigma$  that makes $a,b,c,d$ at most three different colors. So, there is a 4-coloring function of $\Sigma$, say $f$, that makes $f(a)=1$, $f(b)=2$, $f(c)=3$, $f(d)=2$. Notice that there exists either a 1-3 Kempe chain connecting $a$ and $c$, or a 2-4 Kempe chain connecting $b$ and $d$. However, the existence of that 2-4 Kempe chain will cause vertex $f(c)=1$ by vertex-coloring-switching applying on the 1-3 connected component containing $c$. And then we get a new 4-coloring of $\Sigma$  making $a,b,c,d$ only two colors; so a contradiction! Thus, the only possible is 1-3 Kempe chain connecting $a$ and $c$, and $b$ and $d$ belong to two different 2-4 connected components. Now we back to the 4-colorable MGP $\Sigma'\cup{ac}$. Without loss of generality, we have 4-coloring function $f':V(\Sigma)\rightarrow\{1,2,3,4\}$ with $f'(a)=1$, $f(b)=2$, $f(c)=3$, $f(d)=x$ where $x=2$ or $4$. If $x=2$ then once again $EP=\Sigma\cup\Sigma'$ is 4-colorable; so a contradiction! If $x=4$, let us go back to $\Sigma$ and do vertex-coloring-switching on the 2-4 connected component containing $d$ and make $f(d)=4$, i.e., we makes $a,b,c,d$ colored by $1, 2, 3, 4$ respectively. Once again $EP=\Sigma\cup\Sigma'$  is 4-colorable; so a contradiction.  For all these contradictions we conclude that $a,b,c,d$ induce a single diamond in $EP$ which is equivalent to the fact that either $ac$ or $bd$ is an edge of $EP$.   
   \end{proof}
At the end of Part II, we leave a question:  What are the possible shapes for a 5-cycle in an $EP$?   

\bibliographystyle{amsplain}

\end{document}